\documentclass[12pt]{amsart}

\usepackage{fullpage,url,amssymb,enumerate,colonequals}
\usepackage[all]{xy}
\usepackage{color}
\usepackage[OT2,T1]{fontenc}

\newcommand{\F}{\mathbb{F}}
\newcommand{\PP}{\mathbb{P}}
\newcommand{\Q}{\mathbb{Q}}
\newcommand{\R}{\mathbb{R}}
\newcommand{\Z}{\mathbb{Z}}
\newcommand{\kbar}{{\overline{k}}}
\newcommand{\lrs}[1]{(\!(#1)\!)}
\newcommand{\calC}{\mathcal{C}}
\newcommand{\calJ}{\mathcal{J}}
\newcommand{\calO}{\mathcal{O}}
\newcommand{\calQ}{\mathcal{Q}}
\newcommand{\calX}{\mathcal{X}}

\DeclareMathOperator{\Char}{char}
\DeclareMathOperator{\cl}{cl}
\DeclareMathOperator{\im}{im}

\DeclareMathOperator{\res}{res}
\DeclareMathOperator{\Sel}{Sel}

\newcommand{\tors}{{\operatorname{tors}}}

\newcommand{\injects}{\hookrightarrow}
\newcommand{\surjects}{\twoheadrightarrow}
\newcommand{\rholog}{\rho \log}
\newcommand{\To}{\longrightarrow}

\numberwithin{equation}{section}

\newtheorem{theorem}{Theorem}[section]
\newtheorem{lemma}[theorem]{Lemma}
\newtheorem{corollary}[theorem]{Corollary}
\newtheorem{proposition}[theorem]{Proposition}

\theoremstyle{definition}

\newtheorem{conjecture}[theorem]{Conjecture}

\newtheorem{algo}[theorem]{Algorithm}

\theoremstyle{remark}
\newtheorem{remark}[theorem]{Remark}

\definecolor{darkgreen}{rgb}{0,0.5,0}
\usepackage[
        colorlinks, citecolor=darkgreen,
        backref,
        pdfauthor={Michael Stoll}, % add other authors
        pdftitle={Chabauty without the Mordell-Weil group},
]{hyperref}
\usepackage[alphabetic,backrefs,lite]{amsrefs} % for bibliography

\begin{document}

\title{Chabauty without the Mordell-Weil group}

\author{Michael Stoll}
\address{Mathematisches Institut,
         Universit\"at Bayreuth,
         95440 Bayreuth, Germany.}
\email{Michael.Stoll@uni-bayreuth.de}
\urladdr{http://www.mathe2.uni-bayreuth.de/stoll/}

\date{December 7, 2017}

\begin{abstract}
  Based on ideas from recent joint work with Bjorn Poonen,
  we describe an algorithm that can in certain cases determine the set
  of rational points on a curve~$C$, given only the $p$-Selmer group~$S$ of
  its Jacobian (or some other abelian variety $C$ maps to) and the
  image of the $p$-Selmer set of~$C$ in~$S$. The method is more likely
  to succeed when the genus is large, which is when it is usually rather difficult
  to obtain generators of a finite-index subgroup of the Mordell-Weil group,
  which one would need to apply Chabauty's method in the usual way.
  We give some applications, for example to generalized Fermat equations
  of the form $x^5 + y^5 = z^p$.
\end{abstract}

\maketitle

%%%%%%%%%%%%%%%%%%%%%%%%%%%%%%%%%%%%%%%%%%%%%%%%%%%%%%%%%%%%%%%%%%%%%%%%%%%

\section{Introduction}

When one is faced with the task of determining the set of rational points
on a (say) hyperelliptic curve $C \colon y^2 = f(x)$, then the usual way to proceed
is in the following steps. We denote the Jacobian variety of~$C$ by~$J$,
and we assume that $f$ has odd degree, so there is a rational point at infinity
on~$C$, which eliminates possible shortcuts that can be used to show that
a curve does not have any rational points.

\begin{enumerate}[1.]\addtolength{\itemsep}{5pt}
  \item Search for rational points on~$C$. \\[0.5ex]
        This can be done reasonably efficiently for $x$-coordinates whose
        numerator and denominator are at most~$10^5$, say.
        Rational points on curves of genus $\ge 2$ are expected to be fairly
        small (in relation to the coefficients), so the result very likely
        is~$C(\Q)$. It remains to show that we have not overlooked any points.
  \item Compute the $2$-Selmer group~$\Sel_2 J$~\cite{Stoll2001}. \\[0.5ex]
        The `global' part of this computation
        requires arithmetic information related to class group and
        unit group data for the number fields generated by the roots of~$f$.
        If the degrees of the irreducible factors of~$f$ are not too large
        (and the coefficients are of moderate size), then this computation
        is feasible in many cases, possibly assuming the Generalized Riemann
        Hypothesis to speed up the class group computation.
        The `local' part of the computation is fairly easy for the infinite
        place and the odd finite places, but it can be quite involved to
        find a basis of $J(\Q_2)/2 J(\Q_2)$. \\[0.5ex]
        To proceed further, we need the resulting bound~$r$ for the rank of~$J(\Q)$,
        \[ r = \dim_{\F_2} \Sel_2 J - \dim_{\F_2} J(\Q)[2], \]
        to be strictly less than the genus~$g$ of~$C$.
        By work of Bhargava and Gross~\cite{BhargavaGross2013} it is known
        that the Selmer group is small on average, independent of the genus,
        so when $g$ is not very small, this condition is likely to be satisfied.
  \item Find $r$ independent points in~$J(\Q)$. \\[0.5ex]
        We can use the points on~$C$ we have found in Step~1 to get some
        points in~$J(\Q)$. However, it can be quite hard to find further
        points if the points we get from the curve generate a subgroup
        of rank $< r$. There are two potential problems. The first is of
        theoretical nature: the rank of~$J(\Q)$ can be strictly smaller than~$r$,
        in which case it is obviously impossible to find $r$ independent points.
        Standard conjectures imply that the difference between $r$ and the rank
        is even, so we will not be in this situation when we are missing just
        one point. In any case, if we suspect our bound is not tight, we can try
        to use visualization techniques~\cite{BruinFlynn2006} to improve the bound.
        The second problem is practical: some of the generators of~$J(\Q)$ can have
        fairly large height and are therefore likely to fall outside our search
        space. When the genus~$g$ is moderately large, then we also have the
        very basic problem that the dimension of our search space is large. \\[0.5ex]
        To proceed further, we need to know generators of a finite-index
        subgroup~$G$ of~$J(\Q)$.
  \item Fix some (preferably small) prime~$p$ (preferably of good reduction)
        and use the knowledge of~$G$ to compute a basis of the space~$V$ of $\Q_p$-defined
        regular differentials on~$C$ that kill the Mordell-Weil group~$J(\Q)$
        under the Chabauty-Coleman pairing (see for example~\cite{Stoll2006}). \\[0.5ex]
        This requires evaluating a bunch of $p$-adic abelian integrals on~$C$,
        which (in the case of good reduction with $p$ odd) can be done by
        an algorithm due to Bradshaw and Kedlaya and made practical by
        Balakrishnan~\cite{BalakrishnanBradshawKedlaya}.
  \item Find the common zeros of the functions $P \mapsto \int_\infty^P \omega$
        on~$C(\Q_p)$, where $\omega$ runs through a basis of~$V$. \\[0.5ex]
        The rational points are among this set. If there are additional zeros,
        then they can usually be excluded by an application of the Mordell-Weil
        sieve~\cite{BruinStoll2010}.
\end{enumerate}

The most serious stumbling block is Step~3, in particular when the genus~$g$
is of `medium' size (say between 5 and~15), so that Step~2 is feasible, but
we are likely to run into problems when trying to find sufficiently many
independent points in the Mordell-Weil group.

In this paper we propose an approach that circumvents this problem.
Its great advantage is that it uses only the $2$-Selmer group and data
that can be obtained by a purely $2$-adic computation. Its disadvantage
is that it may fail: for it to work, several conditions have to be satisfied,
which, however, are likely to hold in particular when the genus gets large.

Generally speaking, the method tries to use the ideas of~\cite{PoonenStoll2014}
(where it is shown that many curves as above have the point at infinity as their
only rational point) to deal with given concrete curves. Section~\ref{S:key}
gives a slightly more flexible version of one of the relevant results
of this paper. In Section~\ref{S:ratpts}, we formulate the algorithm
for hyperelliptic curves of odd degree that is based on this key result.
The method will apply in other situations as well (whenever we are able
to compute a suitable Selmer group), and we plan to work this out in more detail
in a follow-up paper
for the case of general hyperelliptic curves and also for the setting of
`Elliptic Curve Chabauty', where one wants to find the set of $k$-points~$P$
on an elliptic curve~$E$ defined over a number field~$k$ such that $f(P) \in \PP^1(\Q)$,
where $f \colon E \to \PP^1$ is a non-constant $k$-morphism.
One application in the latter setting is given at the end of this paper.
The approach has also already been applied in~\cite{FreitasNaskreckiStoll} to complete the
resolution of the Generalized Fermat Equation $x^2 + y^3 = z^{11}$.

One ingredient of the algorithm is the computation of `halves' of points
in the group~$J(\Q_2)$. In Section~\ref{S:halving} we give a general procedure
for doing this in~$J(k)$, when $J$ is the Jacobian of an odd degree hyperelliptic
curve and $k$ is any field not of characteristic~$2$.
In Section~\ref{S:Ex}, we demonstrate the usefulness of our approach by
showing that the only integral solutions of $y^2 - y = x^{21} - x$ are
the obvious ones.

In Section~\ref{S:FLT}, we show how our method leads to a fairly simple
criterion that implies the validity of Fermat's Last Theorem for a given
prime exponent. This does not lead to any new results, of course, but it
gives a nice illustration of the power of the method. In Section~\ref{S:GFE},
we then apply our approach to the curves $5 y^2 = 4 x^p + 1$. Carrying
out the computations, we can show that the only rational points on these
curves are the three obvious ones, namely $\infty$, $(1, 1)$ and~$(1, -1)$,
when $p$ is a prime $\le 53$ (assuming GRH for $p \ge 23$). A result due
to Dahmen and Siksek~\cite{DahmenSiksek} then implies that the only
coprime integer solutions of the Generalized Fermat Equation
\[ x^5 + y^5 = z^p \]
are the trivial ones (where $xyz = 0$).

As already mentioned, we end with another type of example,
which uses the method in the context
of `Elliptic Curve Chabauty' to show that a certain hyperelliptic curve
of genus~$4$ over~$\Q$ has only the obvious pair of rational points.
The Mordell-Weil rank is~$4$ in this case, so no variant of Chabauty's
method applies directly to the curve.

\subsection*{Acknowledgments}

I would like to thank Bjorn Poonen for useful discussions and MIT for its
hospitality during a visit of two weeks in May~2015, when these discussions
took place. All computations were done using the computer algebra system
Magma~\cite{Magma}.

%%%%%%%%%%%%%%%%%%%%%%%%%%%%%%%%%%%%%%%%%%%%%%%%%%%%%%%%%%%%%%%%%%%%%%%%%%%

\section{The algorithm} \label{S:key}

In this section we formulate and prove a variant
of~\cite{PoonenStoll2014}*{Proposition~6.2}. We then use it to give an
algorithm that can show that
the set of known rational points in some subset~$X$ of the $p$-adic points
of a curve already consists of all rational points contained in~$X$,
using as input only the $p$-Selmer group of the Jacobian of the curve.
The idea behind this goes back to McCallum's paper~\cite{McCallum1994}.

Let $k$ be a number field,
let $C/k$ be a nice (meaning smooth, projective and geometrically irreducible)
curve of genus~$g \ge 2$ and let $A/k$ be an abelian
variety, together with a map $i \colon C \to A$
such that $A$ is generated by the image of~$C$ (for example,
$A$ could be the Jacobian of~$C$ and $i$ the embedding given by taking
some $k$-rational point $P_0 \in C(k)$ as basepoint).
Fix a prime number~$p$.
We write $\Sel_p A$ for the $p$-Selmer group of~$A$.
Recall that this is defined as the kernel of the diagonal homomorphism
in the commuting diagram with exact rows
\[ \xymatrix{0 \ar[r] & \dfrac{A(k)}{p A(k)} \ar[r]^-{\delta} \ar[d]
                      & H^1(k, A[p]) \ar[r] \ar[d] \ar[dr]
                      & H^1(k, A)[p] \ar[r] \ar[d]
                      & 0 \\
             0 \ar[r] & \displaystyle\prod_v \frac{A(k_v)}{p A(k_v)} \ar[r]^-{\delta}
                      & \displaystyle\prod_v H^1(k_v, A[p]) \ar[r]
                      & \displaystyle\prod_v H^1(k_v, A)[p] \ar[r]
                      & 0
            }
\]
that is induced by applying Galois cohomology to the short exact sequence
\[ 0 \To A[p] \To A \stackrel{\cdot p}{\To} A \To 0 \]
of Galois modules over~$k$ and over all completions~$k_v$ of~$k$,
so the products in the second row run over all places~$v$ of~$k$.
The vertical maps are induced by $k \injects k_v$.
In particular, for each place~$v$ there is a canonical map
$\Sel_p(A) \to A(k_v)/p A(k_v)$.

We write $k_p = k \otimes_{\Q} \Q_p$; this is the
product of the various completions of~$k$ at places above~$p$.
The set~$C(k_p)$ and the group~$A(k_p)$ can similarly be understood as
products of the sets or groups of $k_v$-points, for the various $v \mid p$.
The inclusion $k \injects k_p$ induces natural maps $C(k) \injects C(k_p)$
and $A(k) \injects A(k_p)$. Let $X \subseteq C(k_p)$ be a subset (for example,
the points in a product of $v$-adic residue disks).
We then have the following commutative diagram of maps.
\[ \xymatrix{ C(k) \cap X \ar@{^(->}[d] \ar@{^(->}[r]
               & C(k) \ar@{^(->}[d] \ar[r]^-{i}
               & A(k) \ar@{^(->}[d] \ar@{->>}[r]^-{\pi}
               & \dfrac{A(k)}{p A(k)} \ar@{^(->}[r]^-{\delta} \ar[d]
               & \Sel_p A \ar[dl]^-{\sigma} \\
              X \ar@{^(->}[r]
               & C(k_p) \ar[r]^-{i}
               & A(k_p) \ar@{->>}[r]^-{\pi_p}
               & \dfrac{A(k_p)}{p A(k_p)}
            }
\]

We introduce some more notation. For $P \in A(k_p)$, we set
\begin{equation} \label{E:defq}
  q(P) \colonequals \{\pi_p(Q) : Q \in A(k_p), \exists n \ge 0 \colon p^n Q = P\}
       \subseteq \frac{A(k_p)}{p A(k_p)} ,
\end{equation}
and for a subset $S \subseteq A(k_p)$, we set $q(S) = \bigcup_{P \in S} q(P)$.
We further define
\[ \nu(P) \colonequals \sup\{n : n \ge 0, P \in p^n A(k_p)\} \in \Z_{\ge 0} \cup \{\infty\} . \]
Note that $\nu(P) = \infty$ is equivalent to $P$ having finite order prime to~$p$; on the
complement of the finite set consisting of such~$P$, $\nu$ and~$q$ are locally constant.

With a view toward further applications, we first state a more general version
of our result, which we will then specialize (see Theorem~\ref{T:key} below).
We remark that $C$ could also be a variety of higher dimension here.

\begin{theorem} \label{T:key-general}
  In the situation described above, fix some subgroup $\Gamma \subseteq A(k)$ and
  assume that
  \begin{enumerate}[\upshape(1)]\addtolength{\itemsep}{3pt}
    \item \label{key-gen-1}
          $\ker \sigma \subseteq \delta(\pi(\Gamma))$, and that
    \item \label{key-gen-2}
          $q\bigl(i(X) + \Gamma\bigr) \cap \im(\sigma) \subseteq \pi_p(\Gamma)$.
  \end{enumerate}
  Then $i\bigl(X \cap C(k)\bigr) \subseteq \bar{\Gamma}
               \colonequals \{Q \in A(k) : \exists n \ge 1 \colon nQ \in \Gamma\}$.
\end{theorem}

\begin{proof}
  Let $P \in X \cap C(k)$. We show by induction on~$n$ that for each $n \ge 0$,
  there are $T_n \in \Gamma$ and $Q_n \in A(k)$ such that $i(P) = T_n + p^n Q_n$.
  This is clear for $n = 0$ (take $T_0 = 0$ and $Q_0 = i(P)$). Now assume that
  $T_n$ and~$Q_n$ exist. Note that $\pi_p(Q_n) \in q(i(P)-T_n)$, so
  \[ \pi_p(Q_n) = \sigma(\delta(\pi(Q_n))) \in q\bigl(i(X) + \Gamma\bigr) \cap \im(\sigma); \]
  by~\eqref{key-gen-2} this implies
  $\sigma(\delta(\pi(Q_n))) \in \pi_p(\Gamma) = \sigma(\delta(\pi(\Gamma)))$.
  This shows that $Q_n \in \Gamma + \ker (\sigma \circ \delta \circ \pi)$.
  By~\eqref{key-gen-1} and since $\delta$ is injective, we have
  \[ \ker (\sigma \circ \delta \circ \pi)
       = \pi^{-1}\bigl(\delta^{-1}(\ker \sigma)\bigr)
       \subseteq \pi^{-1}\bigl(\pi(\Gamma))
       = \Gamma + \ker \pi
       = \Gamma + p A(k) ,
  \]
  which implies that $Q_n \in \Gamma + p A(k)$. So there are $T' \in \Gamma$ and
  $Q_{n+1} \in A(k)$ such that $Q_n = T' + p Q_{n+1}$. We set $T_{n+1} = T_n + p^n T' \in \Gamma$;
  then
  \[ i(P) = T_n + p^n Q_n = T_n + p^n(T' + p Q_{n+1}) = T_{n+1} + p^{n+1} Q_{n+1} . \]
  Now consider the quotient map $\psi \colon A(k) \surjects A(k)/\bar{\Gamma}$.
  Since $\bar{\Gamma}$ is saturated in the finitely generated group~$A(k)$,
  the quotient group is torsion free
  and hence free. Observe that for every $n \ge 0$,
  \[ \psi(i(P)) = \psi(T_n + p^n Q_n) = \psi(T_n) + p^n \psi(Q_n)
                = p^n \psi(Q_n) \in p^n\bigl(A(k)/\bar{\Gamma}\bigr) ,
  \]
  which implies that $\psi(i(P)) = 0$ and so $i(P) \in \bar{\Gamma}$.
\end{proof}

The point of formulating the statement in this way (as compared to~\cite{PoonenStoll2014})
is that we avoid the use of $p$-adic abelian
logarithms, which would require us to compute $p$-adic abelian integrals,
usually with $p = 2$ and in a situation when the curve has bad reduction at~$2$.
Instead, we need to be able to compute~$q(P)$ for a given point~$P$, which comes
down to finding its $p$-division points. At least in some cases of interest,
this approach seems to be computationally preferable.

\begin{remark} \label{R:isogeny}
  Instead of considering multiplication by~$p$,
  we could use an endomorphism~$\psi$ of~$A$ that is an isogeny of degree a power of~$p$
  and such that some power of~$\psi$ is divisible by~$p$ in the endomorphism ring
  of~$A$. We then consider $A(k)/\psi A(k)$, $A(k_p)/\psi A(k_p)$
  and the $\psi$-Selmer group~$\Sel_\psi A$. Note that when $\psi \colon A \to A'$
  is any isogeny whose kernel has order a power of~$p$ and with dual isogeny~$\hat{\psi}$,
  then we can consider $A \times A'$ with the endomorphism
  $\tilde{\psi} \colon (P, P') \mapsto (\hat{\psi}(P'), \psi(P))$,
  which satisfies $\tilde{\psi}^2 = \deg \psi = p^m$, together with the
  morphism $\tilde{\imath} \colon X \to A \times A'$, $P \mapsto (i(P), 0)$.
  Taking $\Gamma \times \{0\}$ in place of~$\Gamma$ and writing the relevant maps as
  \[ \xymatrix{A(k) \ar@{->>}[r]^-{\pi}
                & \dfrac{A(k)}{\hat{\psi}(A'(k))} \ar@{^(->}[r]^-{\delta}
                & \Sel_{\hat{\psi}} A' \ar[r]^-{\sigma} & \dfrac{A(k_p)}{\hat{\psi}(A'(k_p))}
              }
  \]
  and
  \[ \xymatrix{A'(k) \ar@{->>}[r]^-{\pi'}
                & \dfrac{A'(k)}{\psi(A(k))} \ar@{^(->}[r]^-{\delta'}
                & \Sel_{\psi} A \ar[r]^-{\sigma'} & \dfrac{A'(k_p)}{\psi(A(k_p))}
              },
  \]
  the second condition in Theorem~\ref{T:key-general} translates into
  \[ q_A(i(X) + \Gamma) \cap \im(\sigma) \subseteq \sigma(\delta(\pi(\Gamma))) \]
  and
  \[ q_{A'}(\hat{\psi}^{-1}(i(X)) + A'(k)_\tors) \cap \im(\sigma')
        \subseteq \sigma'(\delta'(\pi'(A'(k)_\tors))) .
  \]
\end{remark}

\begin{remark} \label{R:chab}
  The set $X \cap i^{-1}(\bar{\Gamma})$ that contains $C(k) \cap X$ when
  Theorem~\ref{T:key-general} applies can in many cases be determined by the
  usual Chabauty-Coleman techniques; see for example~\cite{Stoll2006}.
  Of course, if $\Gamma$ is finite (and so $\bar{\Gamma} = A(k)_\tors$ is finite
  as well), which is usually the case in applications, then determining
  $X \cap i^{-1}(\bar{\Gamma})$ is essentially trivial.
\end{remark}

We give some indication of how one can compute a set such as $q(P + \Gamma)$,
where $P \in A(k_p)$. We assume that, given $P \in A(k_p)$, we can find
all $Q \in A(k_p)$ such that $pQ = P$.

\begin{lemma} \label{L:qpgamma}
  With the notations used in Theorem~\ref{T:key-general}, fix a complete
  set of representatives $R \subseteq \Gamma$ for~$\Gamma/p\Gamma$.
  Let $P \in A(k_p)$ and set
  $\calQ = \{Q \in A(k_p) : \exists T \in R \colon pQ = P + T\}$.
  Define an equivalence relation on~$\calQ$ via $Q \sim Q' \iff Q-Q' \in \Gamma$,
  and let $\calQ'$ be a complete set of representatives for~$\calQ/{\sim}$.
  Then
  \[ q(P + \Gamma) = \{\pi_p(P+T) : T \in R\}   %\bigl(\pi_p(P) + \pi_p(\Gamma)\bigr)
                     \cup \bigcup_{Q \in \calQ'} q(Q + \Gamma) .
  \]
\end{lemma}

\begin{proof}
  Since $Q + \Gamma = Q' + \Gamma$ whenever $Q \sim Q'$, it is sufficient
  to prove the equality with $\calQ$ in place of~$\calQ'$.
  We first show that the set on the right is contained in the set on the left.
  This is clear for the elements $\pi_p(P+T)$, taking $n = 0$
  in~\eqref{E:defq}. So let now $Q \in \calQ$ and $\xi \in q(Q + \Gamma)$.
  Then there are $n \ge 0$, $T' \in \Gamma$ and $Q' \in A(k_p)$ such that
  $p^n Q' = Q + T'$ and $\pi_p(Q') = \xi$. There is also $T \in \Gamma$
  such that $pQ = P + T$. We then have
  \[ p^{n+1} Q' = p(Q + T') = P + (T + pT') \in P + \Gamma \]
  and so $\xi = \pi_p(Q') \in q(P + \Gamma)$.

  Now we show the reverse inclusion. Let $\xi \in q(P + \Gamma)$, so there
  are $n \ge 0$, $T' \in \Gamma$, $Q' \in A(k_p)$ such that $p^n Q' = P + T'$
  and $\pi_p(Q') = \xi$. There is also some $T \in R$ such that $T - T' = p T''$
  with $T'' \in \Gamma$.
  If $n = 0$, then $\xi = \pi_p(P + T') = \pi_p(P + T)$.
  If $n > 0$, we can write
  \[ P + T = (P + T') + p T'' = p (p^{n-1} Q' + T'') = p Q \]
  with $Q = p^{n-1} Q' + T'' \in \calQ$, and
  $\xi = \pi_p(Q') \in q(Q - T'') \subseteq q(Q + \Gamma)$.
\end{proof}

Whenever $\sup \nu(P + \Gamma) < \infty$, the recursion implied by
Lemma~\ref{L:qpgamma} will terminate, and so the lemma translates into
an algorithm for computing~$q(P + \Gamma)$. We make this condition more explicit.

\begin{lemma} \label{L:nugammainf}
  We write $\cl(\Gamma)$ for the topological closure of~$\Gamma$ in~$A(k_p)$.
  Let $P \in A(k_p)$. Then $\sup \nu(P + \Gamma) = \infty$ if and only if
  there is a point $T \in A(k_p)$ of finite order prime to~$p$ such that
  $P \in T + \cl(\Gamma)$.
\end{lemma}

\begin{proof}
  Let $A(k_p)_1$ be the kernel of reduction (i.e., the product of the kernels
  of reduction of the various $A(k_v)$ with $v$ a place above~$p$) and let
  $m$ denote the exponent of the finite group $A(k_p)/A(k_p)_1$.
  Then for all $P \in A(k_p)$, $p^n m P$ tends to the origin as $n \to \infty$.
  If $\sup \nu(P + \Gamma) = \infty$, then there are arbitrarily large~$n$
  such there exist $\gamma_n \in \Gamma$ and $Q_n \in A(k_p)$ with $P = \gamma_n + p^n Q_n$,
  so $m P - m \gamma_n$ tends to the origin as $n$ gets large.
  Then $P - \gamma_n$ must be close to a point of order~$m$; by restricting
  to a sub-sequence, we find that $P - \gamma_n$ approaches a point~$T \in A(k_p)[m]$.
  Since $T$ is close to $P - \gamma_n = p^n Q$ for arbitrarily large~$n$,
  $T$ must be infinitely divisible by~$p$, so the order of~$T$ is prime to~$p$.
  We clearly have $P \in T + \cl(\Gamma)$.

  For the converse, it suffices to consider $P$ in the closure of~$\Gamma$,
  since the points of finite order prime to~$p$ in~$A(k_p)$ are infinitely $p$-divisible.
  Since any point sufficiently close to the origin is highly $p$-divisible,
  this implies that for each $n \ge 0$ we can find $\gamma_n \in \Gamma$
  and $Q_n \in A(k_p)$ such that $P - \gamma_n = p^n Q_n$. This is equivalent
  to $\sup \nu(P + \Gamma) = \infty$.
\end{proof}

We specialize Theorem~\ref{T:key-general} to the case that $k = \Q$
and $i$ embeds the curve into its Jacobian.
Let $\calC$ be a proper regular model of~$C$ over~$\Z_p$. Then the reduction
map sends $C(\Q_p) = \calC(\Z_p)$ to the set of smooth $\F_p$-points on the special fiber of~$\calC$.
The preimage~$D$ of a smooth $\F_p$-point on the special fiber of~$\calC$ under
the reduction map is called a \emph{residue disk} in~$C(\Q_p)$; see~\cite{PoonenStoll2014}.
It follows from Hensel's Lemma that there is an analytic map $\varphi$ from the open $p$-adic
unit disk to~$C$ such that $D = \varphi(p \Z_p)$. If $p = 2$, then we call the subsets
$\varphi(4 \Z_2)$ and~$\varphi(2 + 4 \Z_2)$ \emph{half residue disks}.

\begin{theorem} \label{T:key}
  Let $C$ be a nice curve over~$\Q$, with Jacobian~$J$.
  Let $P_0 \in C(\Q)$ and take $X \subseteq C(\Q_p)$ to be contained in a residue disk or,
  when $p = 2$ and $J(\Q)[2] \neq 0$, in a half residue disk, and to contain~$P_0$.
  Let $i \colon C \to J$ be the embedding sending $P_0$ to zero.
  With the notation introduced above, assume that
  \begin{enumerate}[\upshape(1)]\addtolength{\itemsep}{3pt}
    \item $\ker \sigma \subseteq \delta(\pi(J(\Q)_\tors))$, and that
    \item $q\bigl(i(X) + J(\Q)_\tors\bigr) \cap \im(\sigma) \subseteq \pi_p(J(\Q)_\tors)$.
  \end{enumerate}
  Then $C(\Q) \cap X = \{P_0\}$.
\end{theorem}

\begin{proof}
  We apply Theorem~\ref{T:key-general} with $k = \Q$, $C$ our curve, $A = J$,
  $i$ as given in the statement and $\Gamma = \bar{\Gamma} = J(\Q)_\tors$.
  This tells us that $i\bigl(C(\Q) \cap X\bigr) \subseteq J(\Q)_\tors$.
  If $p > 2$ or $p = 2$ and $J(\Q)[2] = 0$,
  then the only rational torsion point in the kernel of reduction of~$J(\Q_p)$ is the origin,
  which implies that there cannot be two distinct points in~$X$ both mapping
  to torsion under~$i$. If $p = 2$ and $J(\Q)[2] \neq 0$, then the corresponding
  statement is true if $X$ is a half residue disk, which means that $i(X)$ is
  contained in~$K_2$, the second kernel of reduction; see Section~\ref{S:compute_q} below.
  In both cases, we find that there is at most one rational point in~$X$;
  since $P_0$ is one such point, it must be the only one.
\end{proof}

This leads to the following algorithm.
It either returns {\sf FAIL} or it returns the set of rational points on the curve~$C$.
We refer to~\cite{Stoll2007} for the definition of the $p$-Selmer set $\Sel_p(C)$
of the curve~$C$. Given an embedding~$i$ of~$C$ into its Jacobian~$J$, it can be
interpreted as the subset of~$\Sel_p(J)$ consisting of elements that locally
come from points on the curve.

\begin{algo} \label{Algo-general} \strut \\[5pt]
  \textbf{Input:}~A nice curve~$C$, defined over~$\Q$, with Jacobian~$J$. \\
  \hphantom{\textbf{Input:}}~A point $P_0 \in C(\Q)$, defining an embedding $i \colon C \to J$. \\
  \hphantom{\textbf{Input:}}~A prime number~$p$. \\
  \textbf{Output:} The set of rational points on~$C$, or {\sf FAIL}.

  \begin{enumerate}[1.]\addtolength{\itemsep}{3pt}
    \item Compute the $p$-Selmer group $\Sel_p J$ and the $p$-Selmer set $\Sel_p C$; \\
          $i$ induces a map $i_* \colon \Sel_p C \injects \Sel_p J$.
    \item Search for rational points on~$C$ and collect them in a set $C(\Q)_{\text{known}}$.
    \item \label{algo-g:inj}
          Let $\sigma \colon \Sel_p J \to J(\Q_p)/p J(\Q_p)$ be the canonical map. \\
          If $\ker \sigma \not\subseteq \delta(\pi(J(\Q)_\tors))$, then return {\sf FAIL}.
    \item Let $R$ be the image of~$J(\Q)_\tors$ in $J(\Q_p)/p J(\Q_p)$.
    \item Let $\calX$ be a partition of~$C(\Q_p)$ into residue disks
          whose image in $J(\Q_p)/p J(\Q_p)$ consists of one element and that are contained
          in half residue disks when $p = 2$ and $J(\Q)[2] \neq 0$.
    \item \label{algo-g:check}
          For each $X \in \calX$ do the following:
          \begin{enumerate}[a.]\addtolength{\itemsep}{3pt}
            \item If $X \cap C(\Q)_{\text{known}} = \emptyset$: \\
                  \strut\quad If $\pi_p(X) \subseteq \im(\sigma \circ i_*)$, then return {\sf FAIL}; \\
                  \strut\quad otherwise continue with the next~$X$.
            \item Pick some $P_1 \in C(\Q)_{\text{known}} \cap X$.
            \item Compute $Y = \bigcup_{P \in X, T \in J(\Q)_\tors} q([P-P_1]+T)
                             \subseteq J(\Q_p)/p J(\Q_p)$.
            \item If $Y \cap \im(\sigma) \not\subseteq R$, then return {\sf FAIL}.
          \end{enumerate}
    \item Return $C(\Q)_{\text{known}}$.
  \end{enumerate}
\end{algo}

\begin{proposition} \label{P:correct}
  The algorithm is correct: if it does not return {\sf FAIL}, then it returns
  the set of rational points on~$C$.
\end{proposition}

\begin{proof}
  First note that Step~\ref{algo-g:inj} verifies the first assumption of Theorem~\ref{T:key};
  it returns {\sf FAIL} when the assumption does not hold.
  It is also clear that if the algorithm does not return {\sf FAIL},
  then the set it returns is a subset of~$C(\Q)$.
  We show the reverse inclusion.
  So let $P \in C(\Q)$ be some rational point. There will be some $X \in \calX$
  such that $P \in X$. Then $\pi_p(X)$ is contained in $\im(\sigma \circ i_*)$,
  so since the algorithm did not return {\sf FAIL}, by Step~\ref{algo-g:check}a.
  it follows that $X \cap C(\Q)_{\text{known}} \neq \emptyset$;
  let $P_1 \in X \cap C(\Q)_{\text{known}}$ as in Step~\ref{algo-g:check}b.
  Now by Step~\ref{algo-g:check}d. the second assumption of Theorem~\ref{T:key}
  is satisfied, taking the embedding with base-point~$P_1$. So the theorem
  applies, and it shows that there is only one rational point in~$X$,
  so $P = P_1 \in C(\Q)_{\text{known}}$.
\end{proof}

\begin{remark} \label{R:modifyX}
  We note that in Step~\ref{algo-g:check}, the set $X$ can be further partitioned
  if necessary. If there are several points in~$C(\Q)_{\text{known}}$ that end up
  in the same set~$X$, then the second assumption of Theorem~\ref{T:key} cannot be
  satisfied. But it is still possible that the theorem can be applied to smaller
  disks that separate the points. (If the points are too close $p$-adically,
  this will not work, though. In this case, one could try to use $\Gamma + J(\Q)_\tors$
  in the more general version of the theorem, where $\Gamma$ is the subgroup
  generated by the difference of the two points.)

  There are also cases when it helps to combine several sets~$X$ into one.
  One such situation is when there are points in~$C(\Q_p)$ that differ by a
  torsion point of order prime to~$p$ and such that only one of the corresponding
  sets~$X$ contains a (known) rational point.
\end{remark}

A particularly useful case is when $C$ is hyperelliptic, $A = J$ is the Jacobian
of~$C$, and we consider $p = 2$.
There is an algorithm that computes $2$-Selmer group $\Sel_2 J$, which is feasible in many
cases, compare~\cite{Stoll2001}.
We discuss this further in Section~\ref{S:ratpts} below in the case
when the curve has a rational Weierstrass point at infinity.

Another useful case (using a slightly more general setting)
is related to ``Elliptic curve Chabauty''. Here $A$
is the Weil restriction of an elliptic curve~$E$ over some number field~$k$
such that there is a non-constant morphism $C \to E$ defined over~$k$.
We give an example of this in Section~\ref{S:ellchab}.

%%%%%%%%%%%%%%%%%%%%%%%%%%%%%%%%%%%%%%%%%%%%%%%%%%%%%%%%%%%%%%%%%%%%%%%%%%%%

\section{Computing the image under~$q$ of a disk} \label{S:compute_q}

In this section, we discuss in some detail how to find the image under~$q$
of (the image in~$J$ of) a residue disk of~$C(k_p)$.
The basic idea is that $q$ is locally constant on the curve even near points
where $\nu$ becomes infinite (a variant of this was already used in~\cite{PoonenStoll2014}).
To get a practical algorithm out of this idea, we have to produce an explicit
neighborhood on which $q$ is constant. We will do this first away from the
points where $\nu$ becomes infinite and then also on residue disks centered
at a point where $\nu$ becomes infinite.

Since objects over~$k_p$ are products of objects over the various completions~$k_v$
at places~$v$ above~$p$, we will now work over a fixed such completion.
We fix a non-constant morphism $i \colon C \to J$, where $J$ can be
any abelian variety that is spanned by~$i(C)$.
To ease notation, we write $\pi$ instead of~$\pi_v$ for the
map $J(k_v) \to J(k_v)/pJ(k_v)$.

We assume that we can compute $q(P)$ for any given point $P \in J(k_v)$
that is not (too close to) a point of finite order prime to~$p$.
When $p = 2$ and $C$ is hyperelliptic of odd degree and $J$ is the Jacobian,
this can be done by using the halving algorithm of Section~\ref{S:halving} below:
we compute the image of~$P$ in~$L_2^\square$ and record it; if the image is trivial,
then we compute all halves of~$P$ and apply the same procedure to them. Since
by assumption $P$ is not infinitely 2-divisible, the recursion will eventually
stop with an empty set of points still to be considered.

\medskip

The following is essentially immediate from the definitions.

\begin{lemma} \label{L:qconstJ}
  Let $P_1, P_2 \in J(k_v)$ and assume that $P_1 \equiv P_2 \not\equiv 0 \bmod p^{m+1} J(k_v)$.
  Then $\nu(P_1) = \nu(P_2)$ and $q(P_1) = q(P_2)$.
\end{lemma}

\begin{proof}
  The assumptions imply that $P_1, P_2 \notin p^{m+1} J(k_v)$, so whenever there
  are $Q \in J(k_v)$ and $n \ge 0$ such that $p^n Q = P_1$ or~$P_2$, then $n \le m$.
  Let $P' \in J(k_v)$ such that $P_2 = P_1 + p^{m+1} P'$. Then $p^n Q = P_1$
  implies $p^n (Q + p^{m+1-n} P') = P_2$, so that $\nu(P_2) \ge \nu(P_1)$,
  and by symmetry, we obtain equality.

  Let $\xi \in q(P_1)$; then $\xi = \pi(Q)$ for some~$Q$ such that $p^n Q = P_1$ as above.
  Then $n \le m$ and so $\xi = \pi(Q) = \pi(Q + p^{m+1-n} P') \in q(P_2)$ as well.
  This shows that $q(P_1) \subseteq q(P_2)$; the reverse inclusion follows again by symmetry.
\end{proof}

We write $\calO_v$ for the ring of integers in~$k_v$ and $\varpi$ for a uniformizer.
We abuse notation and write $v \colon k_v^\times \to \Z$ for the additive valuation,
normalized such that $v(\varpi) = 1$. Then $e = v(p)$ is the absolute ramification index of~$K_v$.
We fix a proper regular model~$\calC$ of~$C$ over~$\calO_v$.
Let $\calJ$ be the N\'eron model of~$J$ over~$\calO_v$. For $n \ge 1$,
we denote by
\[ K_n \colonequals \ker\bigl(J(k_v) = \calJ(\calO_v) \to \calJ(\calO_v/\varpi^n\calO_v)\bigr) \]
the `higher kernels of reduction'; $K_n$ is also the group of $\varpi^n \calO_v$-points of the
formal group associated to~$\calJ$.

We now fix a residue disk~$D \subseteq C(K_v)$ with respect to~$\calC$;
we will denote an analytic parameterization $D_0 \to D$ by~$\varphi$, where $D_0$
is the open unit disk. Since $i$ induces a morphism from the
smooth part of~$\calC$ to~$\calJ$, it follows that
\begin{equation} \label{E:diffcurve}
 t, t' \in \varpi\calO_v, \quad v(t-t') \ge m \mathrel{\quad\Longrightarrow\quad}
   i(\varphi(t)) - i(\varphi(t')) \in K_m.
\end{equation}
The formal logarithm converges on~$K_1$ and gives a homomorphism $K_1 \to k_v^{\dim J}$.
Restricted to~$K_m$ with $m > e/(p-1)$, the formal exponential provides an inverse,
so that the formal logarithm gives an isomorphism $K_m \to (\varpi^m \calO_v)^{\dim J}$.
It follows that $p K_m = K_{m+e}$; in particular,
\begin{equation} \label{E:2div}
  K_{ne+m} = p^n K_m \subseteq p^n J(k_v) \quad \text{for all $n \ge 0$.}
\end{equation}
This implies together with~\eqref{E:diffcurve} that for $m$ as above and $n \ge 0$,
\begin{equation} \label{E:2divcurve}
  t, t' \in \varpi\calO_v, \quad v(t-t') \ge ne+m \mathrel{\quad\Longrightarrow\quad}
  i(\varphi(t)) \equiv i(\varphi(t')) \bmod p^n J(k_v) .
\end{equation}
In the following we write $\mu$ for $\lfloor e/(p-1) \rfloor + 1$; this is the
smallest choice of~$m$ in the considerations above. If $k_v = \Q_p$ (or, more generally,
an unramified extension of~$\Q_p$, so that $e = 1$),
then $\mu = 1$ when $p$ is odd, and $\mu = 2$ when $p = 2$.

\begin{corollary} \label{C:qconstC}
  Consider $\varphi \colon D_0 \to D \subseteq C(k_v)$ as above, and let $t_0 \in \varpi\calO_v$
  be such that $\nu(i(\varphi(t_0))) \le n$. Then for all $t$ with $v(t-t_0) \ge e(n+1)+\mu$,
  we have
  \[ \nu(i(\varphi(t))) = \nu(i(\varphi(t_0))) \qquad\text{and}\qquad
     q(i(\varphi(t))) = q(i(\varphi(t_0))).
  \]

  More generally, let $\Gamma \subseteq J(\Q_2)$ be a subgroup.
  If $\max \nu(i(\varphi(t_0)) + \Gamma) \le n$, then for all $t$ with $v(t-t_0) \ge e(n+1)+\mu$,
  we have
  \[ \max \nu(i(\varphi(t)) + \Gamma) = \max \nu(i(\varphi(t_0)) + \Gamma) \qquad\text{and}\qquad
     q(i(\varphi(t)) + \Gamma) = q(i(\varphi(t_0)) + \Gamma).
  \]
\end{corollary}

\begin{proof}
  By \eqref{E:2divcurve}, we have $i(\varphi(t)) \equiv i(\varphi(t_0)) \bmod p^{n+1} J(k_v)$.
  The first claim now follows from Lemma~\ref{L:qconstJ}.
  The second claim follows from the first by considering $i(\varphi(t)) + \gamma$
  for each $\gamma \in \Gamma$ separately, and applying the first claim to the
  shifted embedding $P \mapsto i(P) + \gamma$.
\end{proof}

If the image of the disk~$D$ in~$J$ does not contain a point
of finite order prime to~$p$, then $\nu$ will be bounded on~$D$.
Corollary~\ref{C:qconstC} then provides a partition
of~$D$ into finitely many sub-disks such that~$q \circ i$ is constant on each of them.
In this way, we can compute~$q(i(D))$.
In a similar way, this allows us to compute $q(i(D) + \Gamma)$ if
$i(D)$ does not meet $\cl(\Gamma) + J(k_v)[p']$, where $G[p']$ denotes
the subgroup of an abelian group~$G$ consisting of elements of finite order prime to~$p$;
compare Lemma~\ref{L:nugammainf}.

\medskip

We now consider the situation when $D$ contains a point~$P_0$ such that
$i(P_0) \in J(k_v)[p']$. In this case, the result above will not produce
a finite partition into sub-disks, so we need to have an explicit estimate
for the size of the pointed disk around~$P_0$ on which~$q \circ i$ is constant.
Without loss of generality, $i(P_0) = 0$.
We also assume that $\varphi(0) = P_0$, so that $i(\varphi(0)) = 0 \in J$.

In the following, we write $n_\tors$ for the smallest $n \ge 0$
such that $J(k_v)[p^\infty] \subseteq J[p^n]$. In other words, $p^{n_\tors}$
is the exponent of the $p$-power torsion subgroup~$J(k_v)[p^\infty]$.

\begin{lemma} \label{L:qp=q2p}
  Let $P \in J(k_v)$.
  \begin{enumerate}[\upshape(1)]
    \item If $n_\tors = 0$, then $\nu(pP) = \nu(P) + 1$
          and $q(P) \subseteq q(pP) \subseteq q(P) \cup \{0\}$.
    \item $\nu(P) > n_\tors$, then $\nu(pP) = \nu(P) + 1$ and $q(pP) = q(P)$.
  \end{enumerate}
\end{lemma}

\begin{proof}
  Since $p^n Q = P$ implies $p^{n+1} Q = pP$, the inclusion $q(P) \subseteq q(pP)$
  is clear, as is the inequality $\nu(pP) \ge \nu(P) + 1$, for arbitrary~$P$.

  First assume that $n_\tors = 0$. Consider $\xi \in q(pP)$, so there are
  $Q \in J(k_v)$ and $n \ge 0$ such that $p^n Q = p P$ and $\pi(Q) = \xi$.
  If $n = 0$, then $\xi = \pi(Q) = \pi(pP) = 0$. If $n \ge 1$, then we must have $p^{n-1} Q = P$
  (since there is no nontrivial $p$-torsion), so $\xi = \pi(Q) \in q(P)$.
  Taking $n = \nu(pP)$ shows that $\nu(P) \ge \nu(pP) - 1$.

  Now assume that $\nu(P) > n_\tors$ and
  write $P = p^{n_\tors+1} P_0$ for $P_0 \in J(k_v)$.
  We first show that $\pi(J(k_v)[p^\infty]) \subseteq q(P)$. For this, let
  $T \in J(k_v)[p^\infty] = J(k_v)[p^{n_\tors}]$.
  Then $p^{n_\tors}(T + p P_0) = P$, so
  $\pi(T) = \pi(T + pP_0) \in q(P)$ by~\eqref{E:defq}.

  To show that $q(pP) \subseteq q(P)$, let $\xi \in q(pP)$, so there are
  some $Q \in J(k_v)$ and~$n \ge 0$ with $\pi(Q) = \xi$
  such that $p^n Q = p P = p^{n_\tors+2} P_0$. If $n \le n_\tors + 1$,
  then it follows that $Q = p^{n_\tors+2-n} P_0 + T$ with $T \in J(k_v)[p^\infty]$,
  so $\xi = \pi(Q) = \pi(T) \in q(P)$ by the argument above. If $n \ge n_\tors + 2$,
  then $p^{n-n_\tors-2} Q = P_0 + T$ with $T \in J(k_v)[p^{n_\tors}]$, and therefore
  $p^{n-1} Q = p^{n_\tors+1} P_0 = P$, so $\xi = \pi(Q) \in q(P)$.
  Carrying out this argument with $n = \nu(pP)$ and a suitable~$Q$,
  we also get that $\nu(P) \ge \nu(pP) - 1$.
\end{proof}

For $m \ge 1$ we define
\[ N(m) = 1 + \min \Bigl\{\Bigl\lfloor\frac{km-v(k)-\mu}{e}\Bigr\rfloor : k \ge 2\Bigr\}. \]
Then $N(m) e \ge 2 m - a$ for some constant~$a$.

\begin{lemma} \label{L:2iphit=iphi2t}
  Assume that $v(t) = m \ge 1$. Then
  \[ p \cdot i(\varphi(t)) \equiv i(\varphi(pt)) \bmod p^{N(m)} J(k_v) . \]
\end{lemma}

\begin{proof}
  In terms of formal group coordinates, we can write
  \[ \log_J i(\varphi(t)) = c_1 t + \frac{c_2}{2} t^2 + \frac{c_3}{3} t^3 + \ldots \]
  with $c_1, c_2, c_3, \ldots \in \calO_v^{\dim J}$. We find that
  \begin{align*}
    \log_J\bigl(p i(\varphi(t)) - i(\varphi(pt))\bigr)
      &= p \log_J i(\varphi(t)) - \log_J i(\varphi(pt)) \\
      &= c_2 \frac{p-p^2}{2} t^2 + c_3 \frac{p-p^3}{3} t^3 + c_4 \frac{p-p^4}{4} t^4 + \ldots
       \in (\varpi^{N(m)e + \mu} \calO_v)^{\dim J},
  \end{align*}
  by the definition of~$N(m)$. We have that
  \[ p \cdot i(\varphi(t)) - i(\varphi(pt)) \in p K_m + K_{m+e} = K_{m+e} \subseteq K_{\mu}, \]
  so we are in the domain of the isomorphism induced by the formal logarithm,
  which allows us to conclude that
  $p i(\varphi(t)) - i(\varphi(pt)) \in K_{N(m)e + \mu}$.
  The claim then follows from~\eqref{E:2div}.
\end{proof}

\begin{corollary} \label{C:2hyp}
  If we have $p = 2$ and $e = 1$ (which is the case when $k_v = \Q_2$)
  in the situation of Lemma~\ref{L:2iphit=iphi2t}, then
  \[ 2 i(\varphi(t)) \equiv i(\varphi(2t)) \bmod 2^{2m-2} J(\Q_2) . \]
  If in addition $C$ is hyperelliptic, $\varphi(0)$ is a Weierstrass point
  and $\varphi(-t) = \iota(\varphi(t))$, where $\iota$ is the hyperelliptic
  involution, then
  \[ 2 i(\varphi(t)) \equiv i(\varphi(2t)) \bmod 2^{3m-1} J(\Q_2) . \]
\end{corollary}

\begin{proof}
  If $p = 2$ and $e = 1$, then $\mu = 2$ and so $N(m) = 2m-2$ in Lemma~\ref{L:2iphit=iphi2t}
  (the minimum is attained for $k = 2$).

  Under the additional assumptions on~$C$ and~$\varphi$, it follows that
  ${\log_J} \circ i \circ \varphi$ is odd, so that $c_{2n} = 0$ for all $n \ge 1$
  in the proof of Lemma~\ref{L:2iphit=iphi2t}.. We then obtain the better bound in
  the same way as in that proof, noting that we can restrict to odd~$k$
  (which have $v(k) = 0$).
\end{proof}

For our fixed $\varphi$ and~$i$, we define, for $m \ge 1$,
\[ n_m \colonequals \max\{\nu(i(\varphi(t))) : t \in \varpi \calO_v, v(t) = m\} . \]

\begin{lemma} \label{L:nmbound}
  There is some $b \in \Z$ such that $n_m e \le m + b$ for all $m \ge 1$.
\end{lemma}

\begin{proof}
  First note that $i(\varphi(t)) \in K_{m+a} \setminus K_{m+a+1}$
  for some fixed~$a$ when $m$ is sufficiently large, where $a$ is the valuation of~$c_1$
  in the proof of Lemma~\ref{L:2iphit=iphi2t} above.

  Next, let $a'$ denote the $p$-adic valuation of the exponent of the (finite)
  quotient group $J(k_v)/K_{\mu}$. Then for $n \ge \mu$ and $P \in K_n \setminus K_{n+1}$,
  we have $\nu(P) e \le n - \mu + a' e$. To see this, write $P = p^{\nu(P)} Q$ for some
  $Q \in J(k_v)$; assume $\nu(P) e > n - \mu + a' e$. Then $p^{a'} Q$ maps to an element
  of order prime to~$p$ in $J(k_v)/K_{\mu}$, and since $P = p^{\nu(P)-a'} (p^{a'} Q) \in K_{\mu}$,
  it follows that $p^{a'} Q \in K_{\mu}$ (its class in $J(k_v)/K_{\mu}$ has order prime to~$p$
  and a power of~$p$ at the same time, so it must be zero).
  This in turn implies, using~\eqref{E:2div},
  \[ P = p^{\nu(P)} Q = p^{\nu(P) - a'} \cdot (p^{a'} Q) \in K_{\mu + (\nu(P) - a')e}
       \subseteq K_{n+1} ,
  \]
  a contradiction. So $\nu(P) e \le n - \mu + a' e$ as claimed.

  Finally, combining these arguments, we see that $n_m e \le m + (a + a' e - \mu)$
  for large~$m$, which implies the claim.
\end{proof}

\begin{lemma} \label{L:qconstnear0}
  Let $m_0 = 1$ if $n_\tors = 0$ and $m_0 = n_\tors e + \mu + e$ otherwise.
  There is some $m \ge m_0$ such that $N(m) \ge n_m + 1$.
  For any such~$m$, we have have that
  \[ q(i(\varphi(\{t : m \le v(t) < \infty\})))
       = q(i(\varphi(\{t : v(t) = m\}))) \cup \{0\}.
  \]
\end{lemma}

\begin{proof}
  By Lemma~\ref{L:nmbound}, $n_m e \le m + b$ for some~$b$; on the other hand,
  $N(m) e \ge 2 m - a$ for some~$a$, so whenever
  $m \ge a + b + e$, the inequality $N(m) \ge n_m + 1$ holds. Fix such an~$m$ that also
  satisfies $m \ge m_0$. We now show that if $v(t) = m$, then
  \[ \nu(i(\varphi(p^n t))) = \nu(i(\varphi(t))) + n \qquad\text{and}\qquad
     q(i(\varphi(p^n t))) \subseteq q(i(\varphi(t))) \cup \{0\}
  \]
  for all $n \ge 0$, which implies the claim (note that $0 \in q(i(P))$ if $P$
  is sufficiently close to~$P_0$).
  Note that $m \ge n_\tors e + \mu + e$ implies $n_m \ge n_\tors + 1$ by~\eqref{E:2divcurve}
  (taking $t' = 0$).
  We proceed by induction on~$n$, the case $n = 0$ being trivial. So consider $n \ge 1$.
  By the inductive assumption, we have
  \[ \nu(i(\varphi(p^{n-1} t))) = \nu(i(\varphi(t))) + n - 1 \le n_m + n - 1
     \quad\text{and}\quad q(i(\varphi(p^{n-1} t))) \subseteq q(i(\varphi(t))) \cup \{0\} .
  \]
  By Lemma~\ref{L:2iphit=iphi2t}, this implies
  $p \cdot i(\varphi(p^{n-1} t)) \equiv i(\varphi(p^n t)) \bmod p^{N(m + (n-1)e)} J(k_v)$,
  and since
  \[ N(m + (n-1)e) \ge N(m) + 2n - 2 \ge n_m + n \ge \nu(i(\varphi(p^{n-1} t))) + 1 \]
  and $n_m \ge n_\tors + 1$ in case $n_\tors > 0$,
  by Lemmas \ref{L:qconstJ} and~\ref{L:qp=q2p} it follows that
  \[ \nu(i(\varphi(p^n t))) = \nu(p \cdot i(\varphi(p^{n-1} t))) = \nu(i(\varphi(p^{n-1} t))) + 1
                            = \nu(i(\varphi(t))) + n
  \]
  and
  \[ q(i(\varphi(p^n t))) = q(p \cdot i(\varphi(p^{n-1} t)))
                          \subseteq q(i(\varphi(p^{n-1} t))) \cup \{0\}
                          \subseteq q(i(\varphi(t))) \cup \{0\} . \qedhere
  \]
\end{proof}

\begin{corollary} \label{C:qconstnear0}
  If $p = 2$ and $e = 1$ in the situation of Lemma~\ref{L:qconstnear0}, then
  we take $m_0 = 1$ if $n_\tors = 0$ and $m_0 = n_\tors + 3$ otherwise.
  There is then some $m \ge m_0$ such that $2m - 3 \ge n_m$.
  For any such~$m$, we have have that
  \[ q(i(\varphi(\{t : m \le v(t) < \infty\})))
       = q(i(\varphi(\{t : v(t) = m\}))) \cup \{0\}.
  \]
  If the curve is hyperelliptic, $P_0 = \varphi(0)$ is a Weierstrass point
  and $\varphi(-t) = \iota(\varphi(t))$, where $\iota$ is the hyperelliptic involution,
  then the condition above can be replaced by $3m - 2 \ge n_m$.
\end{corollary}

\begin{proof}
  This follows again from $\mu = 2$ and $N(m) \ge 2m-2$.
  The improved statement under the additional assumptions follows in the same way
  as for Corollary~\ref{C:2hyp}.
\end{proof}

This now allows us to find $q(i(D))$ when $0 \in i(D)$. First we use Corollary~\ref{C:qconstC}
to determine $q(i(\varphi(\{t : 1 \le v(t) \le m_0-1\})))$.
Then for $m = m_0, m_0 + 1, \ldots$, we find in a similar way
$n_m$ and~$q(i(\varphi(\{t : v(t) = m\})))$. As soon as $n_m + 1 \le N(m)$,
we can stop the computation; we then have
$q(i(D \setminus \{P_0\})) = q(i(\varphi(\{t : 1 \le v(t) \le m\}))) \cup \{0\}$.

We state a special case for later use.

\begin{corollary} \label{C:hypdisk}
  Assume that $C$ is hyperelliptic, of good reduction mod~$2$, and satisfies
  $J(\Q_2)[2] = 0$ and $J(\F_2)[2] = 0$. Let $P_0 \in C(\Q_2)$, choose a parameterization
  $\varphi$ of a residue disk~$D$ centered at~$P_0$ and let $i_{P_0}$ denote the embedding
  of~$C$ into~$J$ sending $P_0$ to~$0$. Then
  \begin{enumerate}[\upshape(1)]\addtolength{\itemsep}{3pt}
    \item $q(i_{P_0}(D)) = q(i_{P_0}(\varphi(2\Z_2^\times \cup 4\Z_2^\times))) \cup \{0\}$, and
    \item if $P_0$ is a Weierstrass point and $\varphi$ satisfies $\varphi(-t) = \iota(\varphi(t))$,
          then \\
          $q(i_{P_0}(D)) = q(i_{P_0}(\varphi(2\Z_2^\times))) \cup \{0\}$.
  \end{enumerate}
\end{corollary}

\begin{proof}
  Since $k_v = \Q_2$, we are in the case $p = 2$ and~$e = 1$.
  The assumptions on $2$-torsion over~$\Q_2$ and over~$\F_2$ imply that $n_\tors = 0$,
  which in turn implies that for $P \in K_m \setminus K_{m+1}$,
  we have $\nu(P) \in \{m-2, m-1\}$,
  for all $m \ge 1$, compare~\cite{PoonenStoll2014}*{Lemma~10.1} and its proof.
  Also, $K_1$ has odd index in~$J(\Q_2)$.
  We can therefore take $b = -1$ in Lemma~\ref{L:nmbound}. Then $m = 2$ is a suitable value
  in Corollary~\ref{C:qconstnear0}. When $P_0$ is a Weierstrass point, then by
  Corollary~\ref{C:qconstnear0} again even $m = 1$ is sufficient.
\end{proof}

We now give a version of Lemma~\ref{L:qconstnear0} that applies when we work
with a subgroup~$\Gamma$ that does not consist of torsion points only.
We restrict here to the case $k_v = \Q_2$; a general statement can be obtained
and proved along the same lines, with changes similar to the statement and
proof of Lemma~\ref{L:qconstnear0}.

We let $\Gamma \subseteq J(\Q_2)$ be a subgroup such that
$\Gamma \cap 2 J(\Q_2) = 2\Gamma$ and such that $\cl(\Gamma)$ is not of finite
index in~$J(\Q_2)$. We define
\[ n_{m,\Gamma} \colonequals \sup\{\nu(i(\varphi(t)) + \gamma)
                                    : \gamma \in \Gamma, t \in 2\Z_2, v(t) = m\} .
\]

\begin{lemma} \label{L:qconstnear0gamma}
  Let $m_0 = 2$ if $n_\tors = 0$ and $m_0 = n_\tors + 3$ otherwise.
  Assume that there is $m \ge m_0$ such that $2m - 3 \ge n_{m,\Gamma}$.
  For any such~$m$, we have have that
  \[ q(i(\varphi(\{t : m \le v(t) < \infty\})) + \Gamma)
       = q(i(\varphi(\{t : v(t) = m\})) + \Gamma) \cup q(\Gamma).
  \]
  If the curve is hyperelliptic, $P_0 = \varphi(0)$ is a Weierstrass point
  and $\varphi(-t) = \iota(\varphi(t))$, where $\iota$ is the hyperelliptic involution,
  then the condition above can be replaced by $3m - 2 \ge n_{m,\Gamma}$.
\end{lemma}

By standard Chabauty-Coleman, the intersection of $i(D)$ with~$\cl(\Gamma)$
is finite. So for $m$ sufficiently large, $i(\varphi(2^m\Z_2))$ will
meet~$\cl(\Gamma)$ only in~$P_0$, hence $n_{m,\Gamma} < \infty$.
So we can hope to find an~$m$ as in the lemma. It is conceivable, however,
that the image of the curve meets~$\cl(\Gamma)$ at~$i(P_0)$ with higher
multiplicity, in which case $n_{m,\Gamma}$ may grow too fast with~$m$.

\begin{proof}
  We show again inductively that if $v(t) = m$, then
  \[ \max \nu(i(\varphi(2^n t)) + \Gamma) = \max \nu(i(\varphi(t)) + \Gamma) + n
     \quad\text{and}\quad
     q(i(\varphi(2^n t)) + \Gamma) \subseteq q(i(\varphi(t)) + \Gamma) \cup q(\Gamma)
  \]
  for all $n \ge 0$ (note that $q(\Gamma) \subseteq q(i(P) + \Gamma)$ if $P$
  is sufficiently close to~$P_0$).
  The case $n = 0$ is trivial. So consider $n \ge 1$.
  By the inductive assumption, we have
  \[ \max \nu(i(\varphi(2^{n-1} t)) + \Gamma)
       = \max \nu(i(\varphi(t)) + \Gamma) + n - 1 \le n_{m,\Gamma} + n - 1 \]
  and
  \[ q(i(\varphi(2^{n-1} t))+ \Gamma) \subseteq q(i(\varphi(t)) + \Gamma) \cup q(\Gamma) .
  \]
  By Corollary~\ref{C:2hyp}, we have
  $2 i(\varphi(2^{n-1} t)) \equiv i(\varphi(2^n t)) \bmod 2^{2m+2n-4} J(\Q_2)$.
  So for every $\gamma \in \Gamma$, we have
  \[ 2 \bigl(i(\varphi(2^{n-1} t)) + \gamma\bigr)
      \equiv i(\varphi(2^n t)) + 2\gamma \bmod 2^{2m+2n-4} J(\Q_2) .
  \]
  Since
  \[ 2m + 2n - 4 \ge n_{m,\Gamma} + n \ge \nu(i(\varphi(2^{n-1} t)) + \gamma) + 1 \]
  and $n_{m,\Gamma} \ge n_\tors + 1$ in case $n_\tors > 0$,
  by Lemmas \ref{L:qconstJ} and~\ref{L:qp=q2p} it follows that
  \[ \nu(i(\varphi(2^n t)) + 2\gamma)
       = \nu(2 (i(\varphi(2^{n-1} t)) + \gamma))
       = \nu(i(\varphi(2^{n-1} t)) + \gamma) + 1
  \]
  and
  \[ q(i(\varphi(2^n t)) + 2\gamma)
       = q(2 (i(\varphi(2^{n-1} t)) + \gamma))
       \subseteq q(i(\varphi(2^{n-1} t)) + \gamma) \cup \{0\}
       \subseteq q(i(\varphi(t)) + \Gamma) \cup q(\Gamma) .
  \]
  Now consider $\gamma \in \Gamma \setminus 2\Gamma$. Since
  $i(\varphi(2^n t)) \in 2^{n+m-2} J(\Q_2)$ and $n+m-2 \ge 1$, we get
  \[ \nu(i(\varphi(2^n t)) + \gamma) = 0 \qquad\text{and}\qquad
     q(i(\varphi(2^n t)) + \gamma) = \{\pi(\gamma)\} \subseteq q(\Gamma) .
  \]
  (We use here that $\gamma \notin 2J(\Q_2)$.)
  Together, these relations imply that
  \[ \max \nu(i(\varphi(2^n t)) + \Gamma) = \max \nu(i(\varphi(t)) + \Gamma) + n \]
  and
  \[ q(i(\varphi(2^n t)) + \Gamma) \subseteq q(i(\varphi(t)) + \Gamma) \cup q(\Gamma) \]
  as claimed.

  The improved statement under the additional assumptions follows again in the same way.
\end{proof}

%%%%%%%%%%%%%%%%%%%%%%%%%%%%%%%%%%%%%%%%%%%%%%%%%%%%%%%%%%%%%%%%%%%%%%%%%%%%

\section{Determining the set of rational points on odd hyperelliptic curves} \label{S:ratpts}

In this section, we specialize the algorithm formulated in Section~\ref{S:key}
to hyperelliptic curves of odd degree over~$\Q$. So let
\[ C \colon y^2 = f(x) \]
be a hyperelliptic curve, given by a squarefree polynomial $f \in \Z[x]$ of
odd degree $2g + 1$ (then $g$ is the genus of~$C$). We understand $C$
to be the smooth projective model of the affine curve given by the equation;
then $C$ is a nice curve. We write $J$ for the Jacobian of~$C$.
For a point $P_0 \in C(\Q)$ (or $C(\Q_2)$), we
let $i_{P_0} \colon C \to J$ denote the embedding that sends
$P_0$ to the origin of~$J$.

To carry out one of the relevant steps, we must compute $q(P)$ for points $P \in J(\Q_2)$
(where $q(P)$ is defined as above with $p = 2$).
The basic strategy for this was explained in Section~\ref{S:compute_q}.
To implement it, we need to be able to divide by~2 in~$J(\Q_2)$. We consider this
problem in Section~\ref{S:halving} below.

We recall the algorithm for computing the $2$-Selmer group of~$J$,
compare~\cites{Schaefer1995,Stoll2001}. Let $C$ be given by the affine equation
$y^2 = f(x)$ with $f \in \Z[x]$ squarefree and of odd degree~$2g+1$,
where $g$ is the genus of~$C$.
Let $L = \Q[x]/\langle f \rangle$ be the associated \'etale algebra and write~$\theta$
for the image of~$x$ in~$L$. If $A$ is any commutative ring, then we
write $A^\square$ for the group $A^\times/(A^\times)^2$ of square classes
in the multiplicative group~$A^\times$ of~$A$.

For any field extension~$k$ of~$\Q$, there is an isomorphism
\begin{equation} \label{E:kummer}
  H^1(k, J[2]) \stackrel{\simeq}{\To}
    \ker\bigl(N_{(L \otimes_{\Q} k)/k} \colon (L \otimes_{\Q} k)^\square \to k^\square\bigr)
\end{equation}
realizing the Galois cohomology group on the left in a concrete way,
and there is the `Cassels map' or `$x-T$' map
\[ \mu_k \colon J(k) \To J(k)/2 J(k) \injects (L \otimes_{\Q} k)^\square \]
that is induced by evaluating $x - \theta$ (multiplicatively) on divisors whose
support is disjoint from the set of Weierstrass points of~$C$.
The image of~$\mu_k$ is contained in the kernel of the norm map above;
$\mu_k$ is the composition of the connecting map $\delta_k \colon J(k) \to H^1(k, J[2])$
induced by the exact sequence of Galois modules
\[ 0 \To J[2] \To J(\kbar) \stackrel{\cdot 2}{\To} J(\kbar) \To 0 \]
with the isomorphism~\eqref{E:kummer}.
We write $\mu = \mu_{\Q}$, and
for $v$ a place of~$\Q$, we write $L_v = L \otimes_{\Q} \Q_v$
(with $\Q_\infty = \R$ as usual) and set $\mu_v = \mu_{\Q_v}$.

Let $\Sigma$ be the set of places of~$\Q$ consisting of~$2$ and the finite
places~$v$ such that the Tamagawa number of~$J$ at~$v$ is even.
The subgroup $L(\Sigma, 2)$ of~$L^\square$ consists of the elements represented
by $\alpha \in L^\times$ such that the fractional ideal generated by~$\alpha$
has the form $I_1^2 I_2$ with $I_2$ supported on the primes above primes in~$\Sigma$.
Then the isomorphic image of~$\Sel_2 J$ in~$L^\square$, which we will identify
with~$\Sel_2 J$, is given by
\[ \Sel_2 J = \{\xi \in L(\Sigma, 2) : N_{L/\Q}(\xi) = \square,
                                      \forall v \in \Sigma \cup \{\infty\}
                                            \colon \rho_v(\xi) \in \im(\mu_v)\} ,
\]
where $\rho_v \colon L^\square \to L_v^\square$ is the canonical map.
There is also the $2$-Selmer set of~$C$, given by
\[ \Sel_2 C = \{\xi \in L(\Sigma, 2) : N_{L/\Q}(\xi) = \square,
                                       \forall v \colon \rho_v(\xi) \in \mu_v(i_\infty(C(\Q_v)))\} .
\]
It is a subset of the $2$-Selmer group. The set of places~$v$ in the condition
can be restricted to the set $\Sigma \cup \{\infty\}$ together with all `small' primes,
where `small' in practice can be rather large; see~\cite{BruinStoll2009}.

Algorithm~\ref{Algo-general}, combined with the representation of $J(k)/2J(k)$
as a subgroup of~$(L \otimes_{\Q} k)^\square$, then leads to the following.

\begin{algo} \label{Algo} \strut \\[5pt]
  \textbf{Input:} A polynomial $f \in \Z[x]$, squarefree and of odd degree $2g+1$. \\
  \textbf{Output:} The set of rational points on $C \colon y^2 = f(x)$, or {\sf FAIL}.

  \begin{enumerate}[1.]\addtolength{\itemsep}{3pt}
    \item Let $J$ denote the Jacobian of $C$. Set $L = \Q[x]/\langle f \rangle$.
    \item Compute $\Sel_2 J$ and $\Sel_2 C$ as a subgroup and a subset of $L^\square$.
    \item \label{algo:inj}
          Let $L_2 = L \otimes_{\Q} \Q_2$; let $r \colon L^\square \to L_2^\square$
          be the map induced by $\Q \to \Q_2$. \\
          If $\ker r \cap \Sel_2 J \not\subseteq \delta(\pi(J(\Q)[2^\infty]))$,
          then return {\sf FAIL}.
    \item Search for rational points on~$C$ and collect them in a set $C(\Q)_{\text{known}}$.
    \item Let $\calX$ be a partition of~$C(\Q_2)$ into residue disks
          whose image in $L_2^\square$ consists of one element and that are contained
          in half residue disks when $J(\Q)[2] \neq 0$.
    \item Let $R$ denote the image of~$J(\Q)[2^\infty]$ in~$L_2^\square$.
    \item \label{algo:imX}
          For each $X \in \calX$ do the following:
          \begin{enumerate}[a.]\addtolength{\itemsep}{3pt}
            \item If $X \cap C(\Q)_{\text{known}} = \emptyset$: \\
                  \strut\quad If $\mu_2(X) \subseteq \Sel_2 C$, then return {\sf FAIL}; \\
                  \strut\quad otherwise continue with the next~$X$.
            \item Pick some $P_0 \in C(\Q)_{\text{known}} \cap X$.
            \item Compute $Y = \mu_2(q(i_{P_0}(X) + J(\Q)[2^\infty])) \subseteq L_2^\square$.
            \item If $Y \cap r(\Sel_2 J) \not\subseteq R$, then return {\sf FAIL}.
          \end{enumerate}
    \item Return $C(\Q)_{\text{known}}$.
  \end{enumerate}
\end{algo}

That the algorithm is correct is a special case of Proposition~\ref{P:correct},
taking into account that torsion points of odd order are infinitely $2$-divisible,
which allows us to replace $J(\Q)_\tors$ with~$J(\Q)[2^\infty]$ at the places
where the latter occurs.

Remark~\ref{R:modifyX} applies in the same way as to the general algorithm.

\begin{remark} \label{R:no2adic}
  We note that the (image of the) Selmer group in~$L^\square$ that is used
  in the algorithm can be replaced by any subgroup~$S$ of~$L^\square$ that contains
  it (and similarly for the Selmer set). For example, we can take
  \[ S = \{\xi \in L(\Sigma, 2) : N_{L/\Q}(\xi) = \square,
                                  \forall v \in \Sigma \cup \{\infty\} \setminus \{2\} \colon
                                      \res_v(\xi) \in \im(\mu_v)\} ,
  \]
  where $\Sigma$ is the set of `bad primes' for $2$-descent on~$J$.
  This leaves out the $2$-adic Selmer condition. Taking it into account
  requires the computation of~$\mu_2(J(\Q_2))$, which is usually the
  most time-consuming step in the local part of the computation of~$\Sel_2 J$.
  We can do without it, since using~$S$ in the algorithm is actually equivalent
  to using~$\Sel_2 J$. To see this, first consider Step~\ref{algo:inj}.
  Since all elements in the kernel of~$r$ satisfy the $2$-adic Selmer condition
  trivially, it follows that $\ker r \cap S = \ker r \cap \Sel_2 J$, so that
  the outcome of~Step~\ref{algo:inj} is the same in both cases.
  Now consider Step~\ref{algo:imX}a. This does not involve $\Sel_2 J$,
  so its outcome is trivially the same in both cases.
  Finally consider Step~\ref{algo:imX}d. If $Y \cap r(S) \not\subseteq R$,
  then there is some $s \in S$ such that $r(s) \notin R$ and $r(s) \in Y$.
  But everything in~$Y$ is of the form $\mu_2(Q)$ for some $Q \in J(\Q_2)$,
  so $Y \subseteq \im(\mu_2)$, which means that $s$ satisfies the $2$-adic Selmer
  condition. This shows that $s \in \Sel_2 J$ and then implies that
  $Y \cap r(\Sel_2 J) \ni r(s) \notin R$, so that the outcome of this step
  is again the same in both cases. The preceding arguments show that
  the algorithm fails on~$S$ if and only if it fails on~$\Sel_2 J$.
  Finally, it is clear that the result will be the same, namely $C(\Q)_{\text{known}}$,
  in both cases when the algorithm does not output {\sf FAIL}

  If $\Sigma \subseteq \{2, p\}$ with $p \not\equiv \pm 1 \bmod 8$, then we
  can also leave out the condition $N_{L/\Q}(\xi) = \square$, since
  then $\Q(\Sigma, 2)$ injects into $\Q_2^\square$, so the norm condition
  is implied by the image under~$r$ being in~$Y$.

  Of course, we can also use a subset of~$L^\square$ that is possibly larger than~$\Sel_2 C$
  instead of the $2$-Selmer set. In fact, this is what we have to do in practice,
  since the computation of the exact $2$-Selmer set usually requires
  taking into account the local conditions for all primes up to some bound
  that is exponential in the genus of~$C$; compare~\cite{BruinStoll2009}.
\end{remark}

If we assume that $C(\Q)_{\text{known}}$ meets every set in~$\calX$,
then the other conditions
required to avoid failure of the algorithm are likely to be satisfied.
This follows from work of Bhargava and Gross~\cite{BhargavaGross2013},
which we use in a similar way as in~\cite{PoonenStoll2014}: the `probability'
that the map $\Sel_2 J \to J(\Q_2)/2 J(\Q_2)$ is injective is at least
$1 - 2^{1-g-\dim_{\F_2} J(\Q_2)[2]}$, and the `probability' that the image
has intersection with~$Y$ contained in~$R$ is at least $1 - (\#(Y/R) - 1) 2^{1-g}$.
Since by the results of~\cite{PoonenStoll2014} $Y$ is usually small
and by~\cite{Stollpreprint} the size of~$Y$ modulo $R$ is uniformly bounded
by some constant times~$g^2$,
there is a very good chance that both conditions are satisfied when $g$ is large.

%%%%%%%%%%%%%%%%%%%%%%%%%%%%%%%%%%%%%%%%%%%%%%%%%%%%%%%%%%%%%%%%%%%%%%%%%%%%

\section{Halving points on odd hyperelliptic Jacobians} \label{S:halving}

In this section we describe an algorithm that computes one `half' or all `halves'
of a point $P \in 2 J(k)$, where $J$ is the Jacobian of a hyperelliptic
curve~$C$ of odd degree over the field~$k$. We assume that $\Char(k) \neq 2$,
so that $C$ can be given by an equation $y^2 = f(x)$ with $f \in k[x]$
squarefree and of odd degree~$2g+1$.

Recall that each point in~$J(k)$ is uniquely represented in the form
$[D - d \infty]$, where $D$ is an effective divisor in general position
defined over~$k$
and $d = \deg D \le g$. An effective divisor~$D$ is said to be \emph{in general position}
if its support does not contain~$\infty$ and $D \not\ge P + \iota(P)$ for
any point $P \in C$, where $\iota \colon C \to C$ is the hyperelliptic involution.

Any effective divisor~$D$ in general position can be described by its
\emph{Mumford representation} $(a, b)$. Here $a \in k[x]$ is a monic polynomial
of degree $d = \deg D$ whose roots are the $x$-coordinates of the points
in the support of~$D$, with appropriate multiplicity (so that $a$ corresponds
to the image of~$D$ under the hyperelliptic quotient map to~$\PP^1$),
and $b \in k[x]$ is another polynomial such that $b(\xi) = \eta$ for any
point $P = (\xi, \eta)$ in the support of~$D$ and satisfying $a \mid f - b^2$.
This polynomial~$b$ is uniquely determined modulo~$a$; in particular, we obtain
a unique representation if we require $\deg(b) < d$. However, it is sometimes
useful to allow additional flexibility, so we will not always insist on this
normalization. In fact, we may also want to allow polynomials~$a$ of larger
degree (this leads to even more non-unique representations, but can be useful
in certain situations).

We will use the notation $(a, b)$ to denote the divisor~$D$, and we will
write $[a,b] = [(a,b) - d \infty]$ for the point on~$J$ corresponding to it.

Let $c$ be the leading coefficient of~$f$. Then in terms of the Mumford
representation, the descent map
$\mu \colon J(k) \to L^\square$ is given by
\[ [a, b] \longmapsto (-c)^{\deg(a)} a(\theta) \cdot (L^\times)^2 \]
if $a$ and~$f$ are coprime. In the general case, write $a_1$ and~$f_1$
for $a$ and~$f$ divided by their (monic) gcd; then
\[ \mu([a, b])
    = \tilde{\mu}(a)
    \colonequals
      (-c)^{\deg(a)} \bigl(a(\theta) - a_1(\theta) f_1(\theta)\bigr)
        \cdot (L^\times)^2 ;
\]
compare~\cite{Schaefer1995}.

Since the kernel of~$\mu$ is~$2 J(k)$, this gives us a way of deciding
whether a point $P \in J(k)$ is divisible by~$2$ in~$J(k)$: this is
equivalent to the existence of a polynomial~$s \in k[x]$ such that
\[ s(\theta)^2 = (-c)^{\deg(a)} \bigl(a(\theta) - a_1(\theta) f_1(\theta)\bigr); \]
equivalently,
\[ s^2 \equiv (-c)^{\deg(a)} (a - a_1 f_1) \bmod f . \]

We will now state a result that shows how to compute a point~$Q \in J(k)$
such that $2Q = P$, given such a polynomial~$s$.

Note that when $a = a_1^2 a_2$, then $P = [a, b]$ is divisible by~$2$ if and only
if $P_2 = [a_2, b]$ is, and each point $Q$ such that $2Q = P$ has the form
$Q = Q_1 + Q_2$ where $Q_2$ satisfies $2Q_2 = P_2$ and $Q_1 = [a_1, b]$.
So we can assume that $a$ is squarefree.

\begin{proposition} \label{P:halve}
  Let $a \in k[x]$ be monic and squarefree, of degree $\le 2g+1$.
  Let $d$ denote $\gcd(a, f)$, so that $a = d a_1$ and $f = d f_1$ as above.
  Suppose we have $b, s \in k[x]$ with
  \[ f \equiv b^2 \bmod a \qquad\text{and}\qquad
     (-c)^{\deg(a)} (a - a_1 f_1) \equiv s^2 \bmod f ,
  \]
  so that $[a,b] \in 2 J(k)$.
  For polynomials $u$, $v$ and~$w$, consider the following system of congruences:
  \begin{equation} \label{E:congruences}
    v d \equiv w s \bmod f_1, \quad
    v d \equiv u b \bmod a_1, \quad
    u f_1 \equiv w s \bmod d.
  \end{equation}
  Then this system has a nontrivial solution $(u, v, w)$ with $w$ monic such that
  \begin{equation} \label{E:restrictions}
    \deg(u) < \deg(a)/2,\quad \deg(v) \le g + \deg(a)/2 - \deg(d)
    \quad\text{and}\quad \deg(w) \le g.
  \end{equation}
  Each such solution satisfies the relation
  \begin{equation} \label{E:rel}
    u^2 f_1 = d v^2 - (-c)^{\deg(a)} a_1 w^2.
  \end{equation}
  Now assume that $(u, v, w)$ is a solution such that $w$ has minimal degree.
  Let $d_1 = \gcd(u, w)$; then $d_1$
  divides $f a_1$ and~$v$.
  Write $d_1 = d_f d_a$ with $d_f = \gcd(d_1, f)$ and $d_a = \gcd(d_1, a_1)$.
  Set $w_1 = w/d_1$, $u_1 = u/d_1$, $v_1 = v/d_1$
  and let $r \in k[x]$ be such that
  \[ r u_1 \equiv -v_1 d \bmod w_1 d_a \qquad\text{and}\qquad r \equiv 0 \bmod d_f . \]
  Then $Q = [w, r]$ satisfies $P = 2Q$.

  If $Q$ and~$Q'$ are computed starting from $s$ and~$s'$ such that
  $s' \not\equiv \pm s \bmod f$, then $Q$ and~$Q'$ are distinct.
\end{proposition}

\begin{proof}
  First note that, since $f$ is squarefree, we have that $d$ and~$f_1$ are
  coprime. Also, $d$ and~$a_1$ are coprime, since a divisor in general position
  contains no ramification point with multiplicity~$2$ or more. So $f_1$, $a_1$
  and~$d$ are coprime in pairs and squarefree. The fact that $a$ divides~$f - b^2$ implies
  that $d$ divides~$b$ and that $d$ is also the gcd of $a$ and~$b$.

  The first claim is that the system of congruences has a nontrivial solution
  when the degrees of the polynomials are bounded as stated. To see this, note
  that the conditions are linear in (the coefficients of) $u$, $v$ and~$w$, and
  that the total number of coefficients of $u$, $v$ and~$w$ is
  \begin{align*}
    \lceil \deg(a)/2 \rceil + (g + \lfloor \deg(a)/2 \rfloor - \deg(d) + 1) + (g + 1)
      &= 2g + \deg(a) - \deg(d) + 2 \\
      &= \deg(f_1) + \deg(a_1) + \deg(d) + 1 .
  \end{align*}
  On the other hand, the number of linear constraints is $\deg(f_1) + \deg(a_1) + \deg(d)$.
  So there are more variables than constraints, hence nontrivial solutions exist.

  We claim that $w$ cannot be zero in such a solution. Otherwise, the first
  congruence would imply that $f_1$ divides~$v$ (since $f_1$, $a_1$ and~$d$
  are coprime in pairs),
  which for degree reasons (recall that $\deg(a) \le 2g+1$) is only possible when
  $v = 0$. In a similar way, the second congruence would then imply that
  $a_1$ divides~$u$ (since $a_1$ is coprime to~$b$), whereas the third congruence
  implies that $d$ divides~$u$, so $a$ divides~$u$, which is only possible when $u = 0$.
  But then our solution is trivial, a contradiction. So $w \neq 0$, and without
  loss of generality, $w$ can be taken to be monic.

  We show that every solution as above satisfies relation~\eqref{E:rel}.
  Namely, by the first congruence and since $s^2 \equiv (-c)^{\deg(a)} a \bmod f_1$,
  \[ d^2 v^2 = (d v)^2 \equiv (s w)^2 = s^2 w^2
             \equiv (-c)^{\deg(a)} a w^2 = (-c)^{\deg(a)} d a_1 w^2 \bmod f_1,
  \]
  so (since $d$ and~$f_1$ are coprime), the relation holds mod~$f_1$.
  Next, by the second congruence,
  \[ d^2 v^2 = (d v)^2 \equiv (b u)^2 = b^2 u^2 \equiv f u^2 = d f_1 u^2 \bmod a_1, \]
  so (since $d$ and~$a_1$ are coprime), the relation holds mod~$a_1$.
  Finally, by the last congruence,
  \[ u^2 f_1^2 = (u f_1)^2 \equiv (s w)^2 = s^2 w^2 \equiv -(-c)^{\deg(a)} a_1 f_1 w^2
      \bmod d,
  \]
  so (since $d$ and~$f_1$ are coprime again), the relation holds also mod~$d$.
  It follows that it holds mod $f_1 a_1 d$. Since the degrees of all terms
  are strictly less than the degree of $f_1 a_1 d$, equality follows, and
  \eqref{E:rel} is verified.

  We note that the fact shown above that a nontrivial solution has $w \neq 0$
  implies that $w$ determines the solution uniquely. It follows that there is
  in fact a \emph{unique} solution with $w$~monic and $\deg(w)$~minimal.

  Since $d$ is squarefree, \eqref{E:rel}
  implies that the gcd~$d_1$ of $w$ and~$u$ also divides~$v$. We can therefore
  divide all three by this gcd, obtaining $u_1$, $v_1$ and~$w_1$; they satisfy
  \[ u_1^2 f_1 = d v_1^2 - (-c)^{\deg(a)} a_1 w_1^2. \]
  If some irreducible factor~$p$ of~$d_1$ does not divide~$f a_1$, then
  $(u/p, v/p, w/p)$ also satisfy the system of congruences, contradicting
  the minimality of~$\deg(w)$. Now assume that $p^2$ divides~$d_1$ for some
  irreducible polynomial~$p$. Then $p$ divides $f_1$, $a_1$ or~$d$, say $p \mid a_1$
  (the other cases are analogous). Since $a_1$ is squarefree, the congruence
  $v d \equiv u b \bmod a_1$ implies $(v/p) d \equiv (u/p) b \bmod a_1$,
  and so again $(u/p, v/p, w/p)$ satisfy the system of congruences, contradiction.
  So $d_1$ is squarefree and must therefore divide~$a_1 f_1 d$.
  In particular, we can write $d_1 = d_f d_a$ as claimed.

  Note that $(u_1 b)^2 \equiv u_1^2 f \equiv (v_1 d)^2 \bmod a_1$, so that
  $a_1$ divides $(u_1 b - v_1 d)(u_1 b + v_1 d)$. We claim that
  $d_a = \gcd(u_1 b + v_1 d, a_1)$. For this, consider an irreducible
  factor~$p$ of~$d_a$. If $p$ divides $u_1 b - v_1 d$, then $(u/p, v/p, w/p)$
  are a solution, a contradiction. So $p$ must divide $u_1 b + v_1 d$.
  Conversely, if $p$ is any irreducible factor of~$a_1$ that divides $u_1 b + v_1 d$,
  then (noticing that $b$ is invertible mod~$a_1$) for $p$ to divide $u b - v d$,
  it must necessarily divide $u$ and~$v$, so $p \mid d_a$.

  $u_1$ is invertible mod~$w_1$, but also mod~$d_a$ (since $u_1$ and~$v_1$
  are coprime as well --- $a_1$ is squarefree --- and $d_a$ is coprime with $f_1$ and~$d$).
  Furthermore, $d_f$ is coprime with $w_1$ (and of course also with~$d_a$),
  for essentially the same reason. Therefore
  a polynomial~$r$ exists such that $u_1 r \equiv -v_1 d \bmod w_1 d_a$
  and $r \equiv 0 \bmod d_f$.

  Now we consider the function
  \[ \phi = u(x) y - v(x) d(x) = d_f(x) d_a(x) \left(u_1(x) y - v_1(x) d(x)\right) \]
  on~$C$. Its divisor of zeros is
  \begin{align*}
    2 (d_f, 0) &+ \bigl((d_a, b) + (d_a, -b)\bigr)
                + \bigl((d, 0) + (d_a, -b) + (a_1/d_a, b) + 2(w_1, -r)\bigr) \\
      &= (a_1, b) + (d, 0) + 2\bigl((d_f, 0) + (d_a, -b) + (w_1, -r)\bigr) \\
      &= (a, b) + 2(w, -r) .
  \end{align*}
  To see this, note that the norm in~$k[x]$ of the last factor of~$\phi$
  is $u_1^2 f - v_1^2 d^2 = (-c)^{\deg(a)} d a_1 w_1^2$ and that
  $u_1 b \equiv v_1 d \bmod a_1/d_a$ and $u_1 b \equiv -v_1 d \bmod d_a$
  (and so also $r \equiv b \bmod d_a$). Setting $Q = [w, r]$, we therefore
  obtain $2 Q = P$.

  We now show that $Q$ determines $s \bmod f$ up to sign. Given $Q = [w, r]$
  such that $2 Q = P$,
  there is a unique function (up to scaling) on~$C$ whose divisor
  is $(a, b) + 2(w, -r) - n \infty$ (where $n = \deg(a) + 2 \deg(w)$);
  this function must then be~$\phi$, which gives us $u$ and~$v$ up to scaling;
  the relation $u^2 f_1 = d v^2 - (-c)^{\deg(a)} a w^2$ then fixes them up to
  a common sign. Write $d_f = d_{f_1} d_d$ with $d_{f_1} = \gcd(d_f, f_1)$
  and $d_d = \gcd(d_f, d)$. In a similar way as above for~$d_a$, one shows
  that $d_{f_1} = \gcd(w_1 s + v_1 d, f_1)$ and $d_d = \gcd(u_1 f_1 + w_1 s, d)$.
  Since $w_1$ is coprime with~$f$, this determines $s \bmod f$ via the congruences
  \begin{gather*}
    w_1 s \equiv v_1 d \bmod f_1/d_{f_1} , \quad
    w_1 s \equiv -v_1 d \bmod d_{f_1} , \\
    w_1 s \equiv u_1 f_1 \bmod d/d_d , \quad
    w_1 s \equiv -u_1 f_1 \bmod d_d .
  \end{gather*}
  A common sign change of $u$ and~$v$ (which is the only ambiguity here)
  results in a sign change of~$s$.
\end{proof}

We can try to use the algorithm implied by Proposition~\ref{P:halve}
over a $p$-adic field.
It will possibly run into precision problems when some of the roots of~$a$
get close to roots of~$f$ (but with the resultant of $a$ and~$f$ still
being nonzero, albeit $p$-adically small) or when the resulting point
is represented by a divisor of lower degree or such that some points
are close to the point at infinity. In practice, however, these problems
occur fairly rarely. A possible remedy in such a case is to replace
$(a, b)$ by another representation $(a', b')$ such that $[a', b'] = [a, b]$
and $\deg(a) > g$. Writing $f - b^2 = a c$, we have
$[c - 2 h b - h^2 a, -b - h a] = [a, b]$ for all polynomials~$h$.
Taking $h$ to be constant already allows us to replace~$a$ by a polynomial~$a'$
that is coprime with~$f$ (and probably we can also arrange $a'$ to be squarefree)
and satisfies $\deg(a') \le g+1$ if $\deg(a) = g$.
Another possibility is to consider points in a residue disk given by
suitable Laurent series, perform the computation on the Laurent series
and then specialize.

\begin{remark}
  In the context of computing $q(P)$, the following observation can be useful.
  Given $P = [a, b]$ with $\deg(a) \le g+1$ and $T = [h, 0] \in J(k)[2]$
  with $h \mid f$ and $\deg(h) \le g$, we can use the method described in
  Proposition~\ref{P:halve} to compute halves of $P + T$ without first
  computing a representation of the sum. For simplicity assume $\gcd(a, f) = 1$
  (this can be arranged, see above). Then $P + T = [ah, b'h]$ where $b'h \equiv b \bmod a$.
  There will be $s_1$ and~$s_2$ such that $s_1^2 \equiv (-c)^{\deg(a)} a h \bmod f/h$
  and $s_2^2 \equiv -(-c)^{\deg(a)} a (f/h) \bmod h$. We obtain the congruences
  \[ v h \equiv w s_1 \bmod f/h, \quad v h \equiv u b \bmod a, \quad
     u (f/h) \equiv w s_2 \bmod h
  \]
  with $\deg(u) < (\deg(a) + \deg(h))/2$, $\deg(v) \le g + (\deg(a) - \deg(h))/2$
  and $\deg(w) \le g$.

  In a similar way, we can divide $P + P'$ by~$2$: let $P = [a,b]$, $P' = [a',b']$
  and assume that $\deg(a) + \deg(a') \le 2g+1$ and that $a$, $a'$ and~$f$ are coprime
  in pairs. Given a polynomial~$s$ such that $s^2 \equiv (-c)^{\deg(a)+\deg(a')} a a' \bmod f$,
  the system to be solved is
  \[ v \equiv w s \bmod f, \quad v \equiv u b \bmod a, \quad v \equiv u b' \bmod a' \]
  with $\deg(u) < (\deg(a) + \deg(a'))/2$, $\deg(v) \le g + (\deg(a) + \deg(a'))/2$
  and $\deg(w) \le g$.
\end{remark}

We mention one implication that can be helpful in applications.

\begin{corollary} \label{C:piofhalf}
  Let $[a', b]$ be the Mumford representation of a point $P \in J(k)$,
  write $a' = a_0^2 a$ with $a$ squarefree and monic and fix a polynomial~$s$
  such that $s^2 \equiv (-c)^{\deg(a)}(a - a_1 f_1) \bmod f$ as above.
  Let $(u, v, w)$ be the
  solution with $w$ monic and of smallest degree of the system~\eqref{E:congruences}
  with the restrictions in~\eqref{E:restrictions}, and let $Q \in J(k)$
  be the associated point such that $2Q = P$. Then
  $\mu(Q) = \tilde{\mu}(a_0) \tilde{\mu}(w)$.
\end{corollary}

\begin{proof}
  This is because according to Proposition~\ref{P:halve}, $Q = [a_0, b] + [w, r]$
  for some~$r \in k[x]$.
\end{proof}

\begin{corollary} \label{C:halve12}
  In the situation of Corollary~\ref{C:piofhalf}, we have the following special cases.
  \begin{enumerate}[\upshape(1)]
    \item If $P = [(\xi,\eta) - \infty] \in 2J(k)$ with $\eta \neq 0$, fix $s \in k[x]$
          such that $s^2 \equiv c (\xi - x) \bmod f$. Let $w$ be the monic polynomial
          of smallest degree such that the residue of smallest degree of~$w s$ modulo~$f$
          has degree~$\le g$. Then the point $Q \in J(k)$ with $2Q = P$ that is
          associated to~$s$ satisfies
          \[ \mu(Q) = \tilde{\mu}(w). \]
    \item If
          $P = [(\xi_1, \eta_1) - (\xi_2, \eta_2)] \in 2J(k)$
          with $\xi_1 \neq \xi_2$ and $\eta_j \neq 0$ for $j \in \{1,2\}$, fix
          $s \in k[x]$ such that $s^2 \equiv (x - \xi_1)(x - \xi_2) \bmod f$.
          Let $w$ be the monic polynomial of smallest degree such that the residue~$v$
          of smallest degree of~$w s$ modulo~$f$ has degree~$\le g+1$ and satisfies
          $\eta_2 v(\xi_1) + \eta_1 v(\xi_2) = 0$. Then the point $Q \in J(k)$
          with $2Q = P$ that is associated to~$s$ satisfies
          \[ \mu(Q) = \tilde{\mu}(w). \]
  \end{enumerate}
\end{corollary}

\begin{proof}
  This follows directly from Corollary~\ref{C:piofhalf}, using that $d = 1$
  (in the notation of Proposition~\ref{P:halve}) in both cases and that $u$ has to
  be constant. In the first case, the congruence $v \equiv u b \bmod a$ is redundant,
  and the system reduces to just $v \equiv w s \bmod f$. In the second case,
  the congruence $v \equiv u b \bmod a$ is equivalent to the condition
  $\eta_2 v(\xi_1) + \eta_1 v(\xi_2) = 0$.
\end{proof}

%%%%%%%%%%%%%%%%%%%%%%%%%%%%%%%%%%%%%%%%%%%%%%%%%%%%%%%%%%%%%%%%%%%%%%%%%%%%

\section{A concrete example} \label{S:Ex}

In this section we use the approach described above to show the following
result.

\begin{theorem} \label{T:ex}
  Assuming GRH, the only integral solutions of the equation
  \[ y^2 - y = x^{21} - x \]
  have $x \in \{-1, 0, 1\}$.
\end{theorem}

We remark that $l = 21$ is the smallest odd exponent such that our method
can be successfully applied to determine the set of integral points on
the curve $y^2 - y = x^l - x$. One can check that for $l \in \{5, 7, 9, 11, 13, 17\}$
the $2$-Selmer rank of the Jacobian is $\ge g = (l-1)/2$, and for
$l \in \{15, 19\}$, the map from $\Sel_2 J$ to $J(\Q_2)/2J(\Q_2)$
is not injective.

We also note that all these curves have a pair of rational points
with $x = 1/4$; these points are of the form $\varphi(2u)$ for a parameterization~$\varphi$
of the residue disk at infinity, where $u \in \Z_2^\times$. For such a point~$P$,
$[P - \infty]$ has nontrivial image in~$J(\Q_2)/2J(\Q_2)$, and this image
is contained in the image of the Selmer group. On the other hand, by
Corollary~\ref{C:hypdisk}, the value of~$q$ on the residue disk of~$\infty$
is given by the values at points of the form~$\varphi(2u)$, so $q(i_\infty(D))$
will meet the image of the Selmer group non-trivially for every disk~$D$ around
infinity, no matter how small. This implies that our approach cannot be used
to show that $\infty$ is the only rational point $2$-adically close to~$\infty$.
This is why we restrict to integral points in the statement of Theorem~\ref{T:ex}.
The result is in fact stronger: it covers all rational solutions whose $x$-coordinate
has odd denominator.

In principle, one could try to deal with the residue disk at infinity using
$\Gamma = \langle \gamma \rangle$ where
$\gamma = [(\frac{1}{4}, \frac{1}{2} + \frac{1}{2^{21}}) - \infty]$, since the three
(known) rational points in the disk map into this group. Unfortunately,
it turns out that $q(i_\infty(P_4))$ meets the image of the Selmer group
outside the image of~$\Gamma$, so that we cannot conclude. Here $P_4 = \varphi(4)$ denotes
a point with $x$-coordinate~$1/4^2$ (we can use a parameterization of the disk
at infinity whose $x$-coordinate is given by $t^{-2}$):
$i_\infty(P_4) + 6\gamma = 2^3 Q$
with $\pi_2(Q) \in \sigma(\Sel_2 J) \setminus \pi_2(\Gamma)$.

\begin{proof}
  Let $C$ denote the curve defined by the equation $y^2 - y = x^{21} - x$,
  and let $J$ be its Jacobian. Note that $C$ is isomorphic to the curve given by
  $y^2 = 4x^{21} - 4x + 1 =: f(x)$; let $L = \Q[x]/\langle f \rangle$.
  We compute a group~$S \subset L^{\square}$ containing~$\Sel_2 J$ using the algorithm
  described in~\cite{Stoll2001}.
  The discriminant of~$f$ is $-2^{40}$~times the product of six distinct odd
  primes. This implies that $2$ is the only `bad' prime for $2$-descent,
  so that the image of the Selmer group is contained in~$L(\{2\}, 2)$.
  Since $L$ is totally ramified at~$2$, we can reduce this to $S = L(\emptyset, 2)$
  (if $\xi$ represents an element of $L(\{2\}, 2)$ and $N_{L/\Q}(\xi)$
  is a square, then the ideal generated by~$\xi$ must be a square).
  The class group of~$L$ turns out to be trivial, so that
  $L(\emptyset, 2) = \calO_L^\square$, but we do not need this fact.
  We do need to compute $L(\emptyset, 2)$ and explicit generators of it, though.
  This is where we use GRH to make the computation feasible in reasonable time.
  We check that the map $S \to L_2^\square$ is injective.

  The curve has good reduction mod~$2$, and $J(\F_2)$ and $J(\Q_2)$ both
  have no elements of order~$2$. Up to the action of the hyperelliptic involution,
  there are two residue disks with $2$-adically integral $x$-coordinates;
  we can center them at the rational points $(0, 0)$ and~$(1, 0)$, respectively.
  By~\cite{Stoll2001}*{Lemma~6.3}, it follows that the image in~$L_2^\square$
  of a point $P \in C(\Q_2)$ with $x(P) \in \Z_2$
  depends only on $x \bmod 4$. We check that the image in~$L_2^\square$
  of the points with $x(P) \equiv 2 \bmod 4$ is not in the image of~$S$.
  This shows that any ($2$-adically) integral point $P \in C(\Q)$
  must have $x(P) \equiv -1, 0$ or~$1 \bmod 4$. We consider each of the corresponding
  (pairs of) half residue disks separately. Let $P_0$ be one of the points
  $(-1, 0)$, $(0, 0)$ or~$(1, 0)$ on~$C$ and let $D$ be the disk around~$P_0$
  consisting of points~$P$ with $x(P) \equiv x(P_0) \bmod 4$ and $y(P) \equiv 0 \bmod 2$.
  By Corollary~\ref{C:hypdisk} (note that the disk~$D$ corresponds to $m \ge 2$
  in terms of the maximal residue disk around~$P_0$), we have
  \[ q(i_{P_0}(D)) = q(i_{P_0}(\varphi(4 \Z_2^\times))) , \]
  where $\varphi$ is a parameterization of the residue disk containing~$P_0$
  such that $\varphi(0) = P_0$ and $D = \varphi(4 \Z_2)$.
  By Lemma~\ref{L:qconstJ} and since $\nu(i_{P_0}(\varphi(4 u))) = 1$
  for some $u \in \Z_2^\times$ (as becomes apparent in the course of the computation),
  it is sufficient to consider $\varphi(4)$ and~$\varphi(-4)$.
  So we compute the (unique) half of $i_{P_0}(P)$ for each point~$P \in D$
  such that $x(P) = x(P_0) \pm 4$;
  we find that its image in~$L_2^\square$ is nontrivial (and does not depend on the sign)
  and is not contained in the image of~$S$. By Theorem~\ref{T:key}
  this now implies that $D \cap C(\Q) = \{P_0\}$, for each of the three points.
  So we obtain the result that
  \[ C(\Q) \cap C(\Z_2) = \{(-1, 0), (-1, 1), (0, 0), (0, 1), (1, 0), (1, 1)\} \]
  as claimed.
\end{proof}

%%%%%%%%%%%%%%%%%%%%%%%%%%%%%%%%%%%%%%%%%%%%%%%%%%%%%%%%%%%%%%%%%%%%%%%%%%%%

\section{An application to Fermat's Last Theorem} \label{S:FLT}

In this section we apply the criterion that is given by the algorithm
in Section~\ref{S:ratpts} to a certain family of hyperelliptic curves
that are related to Fermat curves. This leads to a criterion for Fermat's
Last Theorem to hold for a given prime~$p$. Of course, FLT has been proved
in general by Wiles~\cite{Wiles1995,TaylorWiles}, so this will not produce a new result.
On the other hand, it shows that the method does work in practice. In the next section,
we will deal with a similar family of curves that are related to certain
generalized Fermat equations; our method applies again and does indeed solve
some new cases of generalized Fermat equations.

Consider
\[ C_l \colon y^2 = f(x) \colonequals 4x^l + 1 \]
with $l = 2g + 1$.
This curve has good reduction at~$2$, since it is isomorphic to $y^2 + y = x^l$.
The reduction has three $\F_2$-points, so there are three residue classes
in~$C_l(\Q_2)$. We also note that $C_l$ has the three obvious rational points
$\infty$, $(0, 1)$ and~$(0, -1)$ and that $[(0, \pm 1) - \infty] \in J_l(\Q)$,
where $J_l$ denotes the Jacobian of~$C_l$, is a point of odd order~$l$.
We note that $J_l(\Q_2)$ and~$J_l(\F_2)$ contain no points of order~$2$.

\begin{corollary} \label{C:imdFLT}
  Let $\varphi \colon D_0 \to D \subseteq C_l(\Q_2)$ be a parameterization of one
  of the three residue disks of~$C_l(\Q_2)$, with $\varphi(0)$ being $\infty$ or~$(0, \pm 1)$.
  Then
  \[ q(i_\infty(D)) = \begin{cases}
                 q(i_\infty(\varphi(2\Z_2^\times \cup 4\Z_2^\times))) \cup \{0\}
                   & \text{if $\varphi(0) = (0, \pm1)$;} \\
                 q(i_\infty(\varphi(2\Z_2^\times))) \cup \{0\}
                   & \text{if $\varphi(0) = \infty$.}
               \end{cases}
  \]
\end{corollary}

\begin{proof}
  This is simply Corollary~\ref{C:hypdisk} specialized to the case at hand.
\end{proof}

We now want to find $q(i_\infty(C_l(\Q_2)))$ in terms of its image in $L_2^\square$
as in Algorithm~\ref{Algo}. To do this, we need a basis for the latter group.
We first note that $f$ is irreducible over~$\Q_2$, so $L_2$ is a field.
Let $\lambda = 2^{1/l}$, then $L_2 = \Q_2(\lambda)$ is totally and tamely ramified
and $\theta = -\lambda^{-2}$ is a root of~$f$.
Clearly, $2 = \lambda^l$. Note that an element of the
form $1 + 4 \alpha \lambda = 1 + \alpha \lambda^{2l+1}$ with $\alpha \in \calO_{L_2}$
is always a square in~$L_2$ (the power series for $\sqrt{1+x}$ converges when the valuation
of~$x$ exceeds that of~$4$). Furthermore,
\[ (1 + \lambda^n)^2
     = 1 + \lambda^{2n} + \lambda^{n+l}
     = (1 + \lambda^{2n})(1 + \lambda^{n+l} + \ldots) ,
\]
which allows us to eliminate factors of the form $1 + \lambda^{2n}$ for $n \le l-1$
when working modulo squares. In this way, we find that the following
elements represent an $\F_2$-basis for $L_2^\square$:
\[ \lambda,\; 1+\lambda,\; 1+\lambda^3,\; \ldots,\; 1+\lambda^{2n+1},\; \ldots,\;
   1+\lambda^{2l-3},\; 1+\lambda^{2l-1},\; 1+\lambda^{2l} .
\]

\begin{lemma} \label{L:FLT_Z}
  The image of $q(i_\infty(C_l(\Q_2)))$ in $L_2^\square$ consists of the classes of
  \[ 1, \quad 1 + \lambda^{l+2},\quad 1 + \lambda^{2l-1},\quad
     \prod_{k\ge 1} (1 + \lambda^{l+2^k}) .
  \]
\end{lemma}

We let $Z$ denote the set consisting of the three nontrivial classes in this image.

\begin{proof}
  We first consider the residue disk~$D_\infty$ around~$\infty$. By Corollary~\ref{C:imdFLT},
  it is sufficient to find $\mu_2(q(i_\infty(\varphi(t))))$
  for $t = 2u$ with $u \in \Z_2^\times$. One choice of~$\varphi$ is
  \[ \varphi(t) = \bigl(t^{-2},
              2 t^{-l}(1 + 2^{-3} t^{2l} - 2^{-7} t^{4l} \pm \ldots)\bigr) .
  \]
  Then $\mu_2(i_\infty(\varphi(2u)))$ is the class of $(2u)^{-2} + \lambda^{-2}$ in~$L_2^\square$.
  We have
  \[ (2u)^{-2} + \lambda^{-2}
       = (2u)^{-2}(1 + u^2 \lambda^{2l-2})
       \sim 1 + \lambda^{2l-2}
       \sim 1 + \lambda^{2l-1} ,
  \]
  where $\sim$ denotes equivalence mod squares, by the relation
  \[ 1 \sim (1 + \lambda^{l-1})^2 = 1 + \lambda^{2l-2} + \lambda^{2l-1}
       \sim (1 + \lambda^{2l-2})(1 + \lambda^{2l-1}) .
  \]
  We conclude that (specifying elements of $L_2^\square$ using representatives in~$L_2^\times$)
  \[ \mu_2(q(i_\infty(D_\infty))) = \{1, 1 + \lambda^{2l-1}\} . \]
  (Compare~\cite{PoonenStoll2014}*{Lemma~10.2}, which says that
  the image of the residue disk at infinity under the $\rholog$ map
  has just one element.)

  Now we consider the residue disk~$D_{(0,1)}$ around~$(0,1)$.
  If $P = (\xi,\eta) \in C_l(\Q_2)$ has integral $x$-coordinate, then we must
  have $\xi \in 2 \Z_2$ (otherwise the right hand side is $5 \bmod 8$ and therefore
  not a square). We can parameterize~$D_{(0,1)}$ by
  \[ \varphi(t) = \left(t, \sqrt{1 + 4 t^l} = 1 + 2 t^l - 2 t^{2l} + \ldots\right) . \]
  Then $\mu_2(i_\infty(\varphi(2u)))$ is the class of
  \[ 2 u + \lambda^{-2} = \lambda^{-2}(1 + \lambda^{l+2} u) \sim 1 + \lambda^{l+2} ; \]
  the latter relation holds when $u$ is a unit. By Corollary~\ref{C:imdFLT},
  we also need to find the image under~$q$ of points given by $t \in 4 \Z_2^\times$,
  so $t = 4u$ with $u \in \Z_2^\times$. In this case (recall that $\theta = -\lambda^{-2}$)
  \[ 4u - \theta
      = (2 \theta^{g+1})^2 (1 - 4 u/\theta)
      = s(\theta)^2
  \]
  where $s \in \Q_2[x]$ is a polynomial of degree $\le l-1$ such that
  \begin{align*}
    s(\theta) &= 2 \theta^{g+1} \sqrt{1 - 4 u/\theta} \\
              &= 2 (\theta^{g+1} - 2 u \theta^{g} - 2 u^2 \theta^{g-1}
                                  - 4 u^3 \theta^{g-2} - \ldots) .
  \end{align*}
  The coefficients of
  \[ \sqrt{1 - 4 x}
        = \sum_{n=0}^\infty 2^{2n} (-1)^n \binom{1/2}{n} x^n
        = 1 - \sum_{n=1}^\infty 2^n \frac{1 \cdot 3 \cdot 5 \cdots (2n-3)}{n!} x^n
  \]
  (except for the constant term)
  all have $2$-adic valuation at least~$1$ (since $v_2(n!) \le n-1$)
  and $\theta^{-1} = - 4 \theta^{2g}$, so
  \[ \frac{1}{2} s(x)
        \equiv x^{g+1} - 2 u x^{g} - 2 u^2 x^{g-1} - \ldots - c_{g+1} u^{g+1}
          \bmod 8 \Z_2[x] ,
  \]
  where $c_{g+1}$ denotes the coefficient of~$x^{g+1}$ in $-\sqrt{1-4x}$.
  Let $w_0(x)$ denote the partial sum of the power series of $(1 - 4 u x)^{-1/2}$
  up to and including the term with~$x^g$, and set
  \[ \tilde{w}(x) = x^g w_0(1/x) = x^g + 2 u x^{g-1} + 6 u^2 x^{g-2} - \ldots . \]
  Then
  \[ \tilde{w}(x) s(x) \equiv 2 x^{2g+1} + (\text{terms up to $x^g$}) \bmod 8 \Z_2[x] , \]
  so $\tilde{w}(x) \equiv w(x) \bmod 8 \Z_2[x]$,
  where $w(x)$ is the monic polynomial of degree~$g$
  such that $w(\theta) s(\theta) \in \Q_2 + \Q_2 \theta + \ldots + \Q_2 \theta^g$.
  Let $Q \in J_l(\Q_2)$ denote the point such that $2 Q = [(4u, *) - \infty]$.
  By Corollary~\ref{C:halve12}, the image of~$Q$ in $L_2^\square$
  is given by the class of
  \begin{align*}
    (-1)^g w(\theta)
      &\sim (-1)^g \tilde{w}(\theta) \\
      &= (-\theta)^g (1 + 2 u \lambda^2 + 6 u^2 \lambda^4 + 20 u^3 \lambda^6
                        + 70 u^4 \lambda^8 + \ldots) \\
      &\sim 1 + 2 \lambda^2 + 6 \lambda^4 + 20 \lambda^6 + 70 \lambda^8 + \ldots \\
      &\sim 1 + \sum_{k=1}^\infty \lambda^{l+2^k} \\
      &\sim (1 + \lambda^{l+2}) (1 + \lambda^{l+4}) (1 + \lambda^{l+8})
              \cdots (1 + \lambda^{l+2^k}) \cdots ,
  \end{align*}
  where the product can be truncated as soon as $2^k > l$. (We have used
  that the valuation of the coefficient of~$x^n$ in~$(1 - 4x)^{-1/2}$ is~$1$
  precisely when $n$ is a power of~$2$.)
\end{proof}

We can generalize this result to certain curves of the form $y^2 = 4 x^l + A$.
Let $A \in \Z$ with $A \equiv 1 \bmod 8$ and consider
\[ C_{l,A} \colon y^2 = 4 x^l +  A . \]
Then $C_{l,A}$ is $\Q_2$-isomorphic to~$C_l = C_{l,1}$, since $A$ is a square and an $l$th~power
in~$\Q_2$. In particular, we still have $L_2 = \Q_2(\lambda)$,
where now $L = \Q[x]/\langle 4 x^l + A \rangle$, and
the image of $q(i_\infty(C_{l,A}(\Q_2)))$ in~$L_2^\square$ is the same
as for~$C_l$, namely $Z \cup \{1\}$.

\begin{proposition} \label{P:hyp1}
  Let $A \in \Z$ satisfy $A \equiv 1 \bmod 8$;
  consider the curve $C_{l,A} \colon y^2 = 4 x^l + A$
  over~$\Q$ with $l = 2g + 1 \ge 5$, with Jacobian~$J_{l,A}$.
  Let $L = \Q[x]/\langle 4 x^l + A \rangle$
  and $L_2 = L \otimes_{\Q} \Q_2 = \Q_2(\lambda)$ with $\lambda = 2^{1/l}$. If
  \begin{enumerate}[\upshape(1)]
    \item the canonical map
          $\Sel_2 J_{l,A} \injects L^\square \to L_2^\square$ is injective and
    \item its image does not meet~$Z$,
  \end{enumerate}
  then $C_{l,A}(\Q) = \{\infty\}$ if $A$ is not a square, and
  $C_{l,A}(\Q) = \{\infty, (0,a), (0,-a)\}$ if $A = a^2$.
\end{proposition}

\begin{proof}
  We apply Theorem~\ref{T:key-general} with $A = J_{l,A}$, $i = i_\infty$,
  $\Gamma = \{0\}$ and $X = C_{l,A}(\Q_2)$.
  By Lemma~\ref{L:FLT_Z}, $Z$ is the set of nontrivial images in~$L_2^\square$
  of elements in $q(i_\infty(C_{l,A}(\Q_2)))$.
  So the assumptions here match the assumptions of Theorem~\ref{T:key-general},
  and we conclude that $i_\infty(C_{l,A}(\Q)) \subseteq \overline{\{0\}} = J_{l,A}(\Q)_{\tors}$.
  Since $C_{l,A}$ has good reduction at~$2$ and $J_{l,A}(\Q)[2]$ is trivial,
  we find that $J_{l,A}(\Q)_{\tors}$ injects into~$J_{l,A}(\F_2)$;
  in particular, $C_{l,A}(\Q)$ will inject into $C_{l,A}(\F_2)$, which has three elements.
  Since each residue class in~$C_{l,A}(\Q_2)$ contains exactly one torsion point
  (namely, $\infty$, $(0, a)$ and $(0, -a)$, respectively, where $a$ is
  a square root of~$A$ in~$\Q_2$), the claim follows.
\end{proof}

It is known that Fermat's Last Theorem holds for a prime~$p \ge 3$
if (and only if) the curve $y^2 = 4 x^p + 1$ has only the obvious three rational points.
So Proposition~\ref{P:hyp1} gives a criterion for FLT for exponent~$p$ to hold, in terms
of the $2$-Selmer group of the Jacobian of this curve.
We can deduce the following criterion.

\begin{proposition} \label{P:FLT1}
  Let $p \ge 5$ be a prime and set $L = \Q(2^{1/p})$ and $L_2 = \Q_2(2^{1/p})$.
  Let $r \colon \calO_L^\square \to \calO_{L_2}^\square$ denote the canonical map.
  If
  \begin{enumerate}[\upshape(1)]
    \item $p^2 \nmid 2^{p-1} - 1$,
    \item the class number of~$L$ is odd, and
    \item $\im(r) \cap Z = \emptyset$ (where $Z$ is as above),
  \end{enumerate}
  then Fermat's Last Theorem holds for the exponent~$p$.
\end{proposition}

\begin{proof}
  Let $f(x) = x^p + 1/4$. Then $f(x-1/4) \equiv x^p \bmod p \Z_p[x]$, and
  the first assumption $p^2 \nmid 2^{p-1} - 1$ implies that the constant term
  is not divisible by~$p^2$. This in turn implies that $C_p \colon y^2 = 4 x^p + 1$
  is regular over~$\Z_p$ and the component group of the N\'eron model
  of the Jacobian~$J$ of~$C$ over~$\Z_p$ is trivial. By~\cite{Stoll2001}*{Lemma~4.5}
  or~\cite{SchaeferStoll2004}*{Proposition~3.2} (which applies equally to
  abelian varieties), the only `bad prime' for the computation of~$\Sel_2 J_p$
  is~$2$. By the second assumption, the class group of~$L$ has odd order
  and therefore trivial $2$-torsion. Together, the previous two sentences
  imply that the isomorphic image of~$\Sel_2 J_p$ in~$L^\square$
  is contained in the subgroup generated by $\calO_L^\square$
  and the image of~$2^{1/p}$. The map to $L_2^\square$
  decomposes as a direct sum of the map~$r$ and an isomorphism of
  $1$-dimensional $\F_2$-vector spaces (since the class of~$\lambda = 2^{1/p}$
  is not contained in the image of the (global or $2$-adic) units).
  We note that $r$ is injective: assume that $u \in \calO_L^\times$
  is a square in~$\calO_{L_2}$. Since $u$ is a unit, the extension $L(\sqrt{u})/L$ is unramified
  at all places not dividing $2$ or~$\infty$. The extension is unramified at~$\infty$,
  since $N_{L/\Q}(u)$ must be~$1$ (it is a $2$-adic square by assumption),
  so the image of~$u$ under the unique real embedding of~$L$ is positive.
  Finally, it is unramified (and even split) at the prime above~$2$.
  Since the class number is odd, there are no nontrivial everywhere unramified
  quadratic extensions of~$L$, hence $u$ must be a square.
  This implies that $\Sel_2 J \to L_2^\square$ is injective.
  Since $Z$ is contained in~$\calO_{L_2}^\square$,
  assumption~(3) implies that $r(\Sel_2 J) \cap Z = \emptyset$ as well.
  We can now apply Proposition~\ref{P:hyp1} and conclude that
  $C_p(\Q) = \{\infty, (0,1), (0,-1)\}$.

  Now let $F_p \colon u^p + v^p + w^p = 0$ denote the projective Fermat
  curve of exponent~$p$. Then there is a non-constant morphism
  \[ \psi \colon F_p \To C_p, \quad
                 (u : v : w) \longmapsto (x, y) =
                   \Bigl(-\frac{uv}{w^2}, 2\frac{u^p}{w^p} + 1\Bigr) .
  \]
  So if $P = (u : v : w) \in F_p(\Q)$, then either $w = 0$ (if $\psi(P) = \infty$)
  or $uv = 0$ (if $\psi(P) = (0, \pm 1)$), so $P$ is a trivial point.
\end{proof}

Note that by Remark~\ref{R:no2adic}, the criterion formulated in the proposition
above is equivalent to what we would obtain when using the $2$-Selmer group~$\Sel_2 J_p$
instead of~$\calO_L^\square$.

We can improve on Proposition~\ref{P:FLT1} a bit. Note that if $u, u' \in \calO_L$ are units
with $u$~positive (in the unique real embedding of~$L$), then the Hilbert symbol
$(u, u')_v$ is~$1$ for all places~$v$ distinct from the place~$\lambda$ above~$2$.
The product formula for the
Hilbert symbol implies that $(u, u')_\lambda = 1$ as well. There are the
two positive global units $\lambda - 1$ and $(1 - \lambda + \lambda^2)/(1 + \lambda)$.
Multiplying the latter by the square $(1+\lambda)^2$, we obtain $1 + \lambda^3$.
So if $u \in \calO_{L_2}^\times$ and we can show that $(\lambda - 1, u)_\lambda = -1$
or $(1 + \lambda^3, u)_\lambda = -1$, then $u$ cannot be in the image of~$\calO_L^\times$.

\begin{lemma} \label{L:Hilbsym}
  We work in~$L_2 = \Q_2(\lambda)$ with $\lambda^l = 2$ as before.
  If $1 \le m < l$, then we have
  \[ (\lambda - 1, 1 + \lambda^{2l - m})_\lambda = -1 \qquad\text{and}\qquad
     (1 + \lambda^3, 1 + \lambda^{2l-m})_\lambda
       = \begin{cases} 1 & \text{if $3 \nmid m$,} \\ -1 & \text{if $3 \mid m$.} \end{cases}
  \]
\end{lemma}

\begin{proof}
  We first consider $\lambda - 1$. Note that
  $(-1, 1 + \lambda^{2l-m})_\lambda = (-1, 1 + 2^{2l-m})_2 = 1$, so we can as well
  work with $(1 - \lambda, 1 + \lambda^{2l-m})_\lambda$. We have for $n \ge (l-1)/2$ that
  \[ (1 + \lambda^n)^2 - (1 - \lambda) (\lambda^n)^2
        = 1 + \lambda^{2n+1} + \lambda^{l+n}
        \sim (1 + \lambda^{2n+1}) (1 + \lambda^{l+n})
  \]
  is a norm from $L_2(\sqrt{1-\lambda})$, which implies that
  \[ (1 - \lambda, 1 + \lambda^{2l-m})_\lambda = (1 - \lambda, 1 + \lambda^{2l-(m+1)/2})_\lambda \]
  when $1 \le m < l$ is odd. For even~$m$, we have
  \[ 1 \sim (1 + \lambda^{l-m/2})^2
       = 1 + \lambda^{2l-m} + \lambda^{2l-m/2}
       \sim (1 + \lambda^{2l-m})(1 + \lambda^{2l-m/2}),
  \]
  which implies that
  \[ (1 - \lambda, 1 + \lambda^{2l-m})_\lambda = (1 - \lambda, 1 + \lambda^{2l-m/2})_\lambda \]
  when $1 \le m < l$ is even. An easy induction then shows that
  \[ (1 - \lambda, 1 + \lambda^{2l-m})_\lambda = (1 - \lambda, 1 + \lambda^{2l-1})_\lambda \]
  for all $1 \le m < l$. Finally, this last symbol is~$-1$: an element is a norm
  from~$L_2(\sqrt{1+\lambda^{2l-1}})$ if and only if it has the form $x^2 - (1 + \lambda^{2l-1})y^2$.
  Substituting $(x,y) \leftarrow (\lambda^{l-1} x + y, y)$ and dividing by $\lambda^{2l-2}$,
  we see that norms have the form $x^2 + \lambda x y - \lambda y^2$. If the norm is
  integral, then $x$ and~$y$ must be in~$\calO_{L_2}$ as well. Considering the equation
  \[ 1 - \lambda = x^2 + \lambda x y - \lambda y^2 \]
  modulo~$\lambda^2$, we see that it has no solution.

  Now we consider $1 + \lambda^3$. For even $1 \le m < l$ we have in the same way as above that
  \[ (1 + \lambda^3, 1 + \lambda^{2l-m})_\lambda = (1 + \lambda^3, 1 + \lambda^{2l-m/2})_\lambda. \]
  For $n \ge (l-1)/2$, we have the norms
  \[ (1 + \lambda^n)^2 - (1 + \lambda^3) (\lambda^n)^2
        = 1 - \lambda^{2n+3} + \lambda^{l+n}
        \sim (1 + \lambda^{2n+3}) (1 + \lambda^{l+n}),
  \]
  leading to
  \[ (1 + \lambda^3, 1 + \lambda^{2l-m})_\lambda = (1 - \lambda, 1 + \lambda^{2l-(m+3)/2})_\lambda \]
  when $1 \le m < l$ is odd. By induction again, we see that
  \[ (1 + \lambda^3, 1 + \lambda^{2l-m})_\lambda
       = \begin{cases}
           (1 + \lambda^3, 1 + \lambda^{2l-1})_\lambda & \text{if $3 \nmid m$,} \\
           (1 + \lambda^3, 1 + \lambda^{2l-3})_\lambda & \text{if $3 \mid m$.}
         \end{cases}
  \]
  Let $a \in L_2$ satisfy $a^2 - a + \lambda^2 = 0$ (such $a$ exist by Hensel's Lemma).
  Then $1^2 + \lambda \cdot 1 \cdot a - \lambda \cdot a^2 = 1 + \lambda^3$ is a
  norm from~$L_2(\sqrt{1+\lambda^{2l-1}})$, so the first symbol is~$1$.
  In a similar way as before, we see that norms from~$L_2(\sqrt{1+\lambda^{2l-3}})$
  are of the form $x^2 + \lambda^2 x y - \lambda y^2$. A consideration modulo~$\lambda^4$
  shows that this can never equal~$1 + \lambda^3$, so the second symbol is~$-1$.
\end{proof}

\begin{corollary} \label{C:FLT2}
  Let $p \ge 5$ be a prime and set $L = \Q(2^{1/p})$ and $L_2 = \Q_2(2^{1/p})$.
  As before, $r \colon \calO_L^\square \to \calO_{L_2}^\square$ denotes the canonical map.
  If
  \begin{enumerate}[\upshape(1)]
    \item $p^2 \nmid 2^{p-1} - 1$,
    \item the class number of~$L$ is odd, and
    \item $4 \nmid \left\lfloor \log_2 p \right\rfloor$ or $z \notin \im(r)$,
          where $z$ is the last element listed in Lemma~\ref{L:FLT_Z},
  \end{enumerate}
  then Fermat's Last Theorem holds for the exponent~$p$.
\end{corollary}

\begin{proof}
  We only have to show that the third condition here implies that
  $\im(r) \cap Z = \emptyset$. By Lemma~\ref{L:Hilbsym}, we have (with $l = p$)
  \[ (\lambda - 1, 1 + \lambda^{p+2})_\lambda = (\lambda - 1, 1 + \lambda^{2p-1})_\lambda = -1, \]
  which implies that the first two elements of~$Z$ can never be images of
  global units. We also have $(\lambda - 1, z)_\lambda = (-1)^{\left\lfloor \log_2 p \right\rfloor}$,
  so we can also rule out~$z$ when $\left\lfloor \log_2 p \right\rfloor$ is odd.
  So we can now assume that $\left\lfloor \log_2 p \right\rfloor \equiv 2 \bmod 4$.
  Then by Lemma~\ref{L:Hilbsym} again, we find that
  $(1 + \lambda^3, z)_\lambda = -1$ (note that every other term in the sequence $(p - 2^k)_k$
  is divisible by~$3$), and we can again rule out~$z$.
\end{proof}

\begin{corollary}
  FLT holds for exponents $5$, $7$, $11$, $13$, $17$, $19$
  and, assuming the Generalized Riemann Hypothesis, also for exponents
  $23$, $29$, $31$, $37$, $41$, $43$, $47$, $53$ and~$59$.
\end{corollary}

\begin{proof}
  We use Magma~\cite{Magma} to check the assumptions (assuming GRH where indicated
  to speed up the computation of the class group).
  It turns out that the class group of $\Q(2^{1/p})$ is trivial for
  all primes considered. We note that $p = 17, 19, 23, 29, 31$ are the only primes~$p$
  up to~$59$ that satisfy $4 \mid \left\lfloor \log_2 p \right\rfloor$, so we
  need a basis of~$\calO_L^\square$ only for these primes; for the remaining
  ones it suffices to know that the class number is odd.
\end{proof}

\begin{remark}
  Computations show that the class group of~$\Q(2^{1/n})$ is trivial
  for \emph{all} $n \le 50$ (assuming GRH for $n \ge 20$), regardless whether $n$ is prime or not.
  According to class group heuristics~\cite{VenkateshEllenberg}*{Section~4.1},
  the $2$-torsion in the class
  group of a number field with unit rank~$u$ should behave like the cokernel
  of a random linear map $\F_2^{n+u} \to \F_2^n$ for large~$n$ (at least in absence
  of special effects leading to systematically occurring elements of order~$2$).
  Such a map is surjective with probability $> 1 - 2^{-u}$, so noting that $u = (p-1)/2$
  in the case of interest, the `probability' that the class number of~$L$ is odd
  for all~$p$ is $> 1 - 2^{-29}$ (assuming we know it for $p \le 59$).
  See also~\cite{HoShankarVarma}.

  We also remark that when the first condition $p^2 \nmid 2^{p-1} - 1$ is not satisfied,
  the criterion does still work when we replace $\calO_L^\square$
  by the larger subgroup $L(\{p\}, 2)$ of~$L^\square$ represented by
  elements generating ideals of the form $I_1^2 I_2$ with $I_2$ supported on
  the ideals above~$p$. In this case, however, we also have to check that the
  map to~$L_2^\square$ is injective.
  A similar remark applies to the case when the class
  group does have even order.
\end{remark}

%%%%%%%%%%%%%%%%%%%%%%%%%%%%%%%%%%%%%%%%%%%%%%%%%%%%%%%%%%%%%%%%%%%%%%%%%%%%

\section{An application to certain generalized Fermat equations} \label{S:GFE}

Recall the following statement.

\begin{proposition}[Dahmen and Siksek, \cite{DahmenSiksek}*{Lemma~3.1 and Proposition~3.3}]
  Let $p$ be an odd prime. If the only rational points on the curve
  \[ C'_p \colon 5 y^2 = 4 x^p + 1 \]
  are the obvious three (namely $\infty$, $(1, 1)$ and~$(1, -1)$), then
  the only primitive integral solutions of the generalized Fermat equation $x^5 + y^5 = z^p$
  are the trivial ones:
  \[ (x,y,z) = \pm (0, 1, 1), \; \pm (1, 0, 1), \; \pm (1, -1, 0) . \]
\end{proposition}

Dahmen and Siksek show that this is true when $p \in \{7, 19\}$ and also
when $p \in \{11, 13\}$, assuming GRH. We will use our approach to extend
the range of primes~$p$ for which it can be shown that $C'_p(\Q)$ has only
the obvious three rational points.

So we now consider the curves~$C'_l$, with $l = 2g+1$ odd, but not necessarily prime.
The corresponding \'etale algebra
is still $L = \Q(\lambda)$ with $\lambda = 2^{1/l}$ (since $C'_l$ is the quadratic twist by~$5$
of $y^2 = 4 x^l + 1$), but the descent map is now given on a point on the Jacobian
with Mumford representation $[a, b]$ by the class of $-5 a(\theta)$ (instead
of~$-a(\theta)$) if the degree of~$a$ is odd.

It is still the case that $C'_l$ has good reduction mod~$2$ (replacing $y$
by~$2y+1$ and dividing by~$4$ gives $5 (y^2 + y) = x^l - 1$) and that there
is no nontrivial $2$-torsion in~$J'_l(\Q_2)$ nor in~$J'_l(\F_2)$, where $J'_l$
denotes the Jacobian of~$C'_l$. We therefore have a statement similar
to Corollary~\ref{C:imdFLT}. Note that we have again three residue disks,
centered at $\infty$, $(1,1)$ and~$(1,-1)$, respectively.

If $P_0 \in C'_l(\Q)$, then we write $D_{P_0}$ for the residue disk centered at~$P_0$.
We let $\varphi_{P_0} \colon D_0 \to D_{P_0}$ be a parameterization of~$D_{P_0}$
such that $\varphi(0) = P_0$ (and such that $i_\infty \circ \varphi_\infty$ is odd).

\begin{corollary} \label{C:imdGFE}
  We have
  \begin{align*}
    q(i_\infty(D_\infty)) &= q(i_\infty(\varphi_\infty(2\Z_2^\times))) \cup \{0\} \\
    q(i_{(1,1)}(D_{(1,1)}))
       &= q(i_{(1,1)}(\varphi_{(1,1)}(2\Z_2^\times \cup 4\Z_2^\times))) \cup \{0\}
  \end{align*}
\end{corollary}

\begin{proof}
  This again follows from Corollary~\ref{C:hypdisk}.
\end{proof}

The main difference with the case discussed in the previous section is that,
if $l \ge 7$,
the two points $(1, \pm 1)$ do not map to points of finite order in~$J'_l$
under the embedding that sends $\infty$ to zero. So from now on, we assume
that $l \ge 7$. Note that the rank of~$J'_5(\Q)$ is zero (the $2$-Selmer
group is trivial), so it is almost immediate that $C'_5(\Q) = \{\infty, (1, \pm 1)\}$.

We first consider the
image of $C'_l(\Q_2)$ in $J'_l(\Q_2)/2J'_l(\Q_2)$ under $q \circ i_\infty$.

\begin{lemma} \label{L:im1GFE}
  In terms of representatives in~$L_2^\times$, we have
  \begin{enumerate}[\upshape(1)]
    \item $\mu_2(q(i_\infty(D_\infty))) = \{1, 1 + \lambda^{2l-1}\}$.
    \item $\mu_2(q(i_\infty(D_{(1,1)})))
             = \{5 (1 + \lambda^2), 5 (1 + \lambda^2 + \lambda^{l+2})\}$.
  \end{enumerate}
\end{lemma}

\begin{proof}
  By Corollary~\ref{C:imdGFE}, we know that
  $q(i_\infty(D_\infty)) = q(i_\infty(\varphi_\infty(2\Z_2^\times))) \cup \{0\}$,
  where we can choose $\varphi_\infty$ such that $x(\varphi_\infty(t)) = 5t^{-2}$.
  So let $u \in \Z_2^\times$, then $\mu_2(i_\infty(\varphi_\infty(2u)))$ is represented by
  \[ 5 \left(\frac{5}{4 u^2} + \lambda^{-2}\right)
      = \left(\frac{5}{2 u}\right)^2 \left(1 + \frac{4 u^2}{5} \lambda^{-2}\right)
      \sim 1 + \lambda^{2l-2}
      \sim 1 + \lambda^{2l-1} .
  \]
  This proves~(1).

  Now let $P \in D_{(1,1)}$. We can choose $\varphi_{(1,1)}$ such that
  $x(\varphi_{(1,1)}(t)) = 1 + t$. If $u \in \Z_2^\times$, then
  $\mu_2(i_\infty(\varphi_{(1,1)}(2u)))$ is represented by
  \[ 5 (1 + 2u + \lambda^{-2})
      \sim 5 (1 + \lambda^2 + u \lambda^{l+2})
      \sim 5 (1 + \lambda^2 + \lambda^{l+2}) ,
  \]
  and for any $u \in \Z_2$, $\mu_2(i_\infty(\varphi_{(1,1)}(4u)))$ is represented by
  \[ 5 (1 + 4u + \lambda^{-2})
      \sim 5 (1 + \lambda^2 + u \lambda^{2l+2})
      \sim 5 (1 + \lambda^2) .
  \]
  This proves~(2).
\end{proof}

Now we consider the embedding $i_{(1,1)}$.

\begin{lemma} \label{L:im2GFE}
  In terms of representatives in~$L_2^\times$, we have
  \[ \mu_2(q(i_{(1,1)}(D_{(1,1)})))
      = \{1, 1 + \lambda^{l+2}/(1 + \lambda^2), \sigma, \sigma'\},
  \]
  where $\sigma = \mu_2(Q)$ for the point $Q \in J_l(\Q_2)$ such that
  $2 Q = \varphi_{(1,1)}(4)$ and $\sigma' = \mu_2(Q')$ where $2 Q' = \varphi_{(1,1)}(-4)$.
\end{lemma}

\begin{proof}
  We make use of Corollary~\ref{C:imdGFE} again, which tells us that it suffices
  to consider points~$P$ with $x$-coordinates $1 + 2u$ or $1 + 4u$, where $u \in \Z_2^\times$.
  If $x = 1 + 2u$, then by the computation in the proof of Lemma~\ref{L:im1GFE},
  we have $\pi_2(i_{(1,1)}(P)) = \pi_2(i_\infty(P)) - \pi_2(i_\infty((1,1)))$, which is
  represented by
  \[ 5 (1 + \lambda^2) \cdot 5 (1 + \lambda^2 + \lambda^{l+2})
       \sim 1 + \frac{\lambda^{l+2}}{1 + \lambda^2} .
  \]
  If $x = 1 + 4u$, then $i_{(1,1)}(P)$ is divisible by~$2$ in~$J'_l(\Q_2)$, so we
  have to look at $\pi_2(Q)$ where $2Q = P$, for suitable values of~$u$. Since
  $i_{(1,1)}(P) \in K_2 \setminus K_3$, we have $\nu(i_{(1,1)}(P)) = 1$, so
  by Corollary~\ref{C:qconstC}, $\pi_2(Q)$ depends only on $u \bmod 4$, so the two
  values $u = 1$ and~$u = -1$ are sufficient.
\end{proof}

In practice, it appears that $\sigma = \sigma'$ in all cases, which would be
implied by the difference of the images of any pair chosen from the relevant points
being divisible by~$4$. We know this difference is in~$K_3$, but we did not exclude
the possibility that it is only divisible by~$2$ and not by~$4$.

We can now formulate a criterion.

\begin{proposition} \label{P:hyp2a}
  Consider $C'_l \colon 5 y^2 = 4 x^l + 1$, with Jacobian~$J'_l$,
  where $l = 2g + 1 \ge 7$ is odd.
  Recall that $L = \Q(2^{1/l})$; let $S \subseteq L^\square$
  be a finite subgroup that contains the image of $\Sel_2 J'_l$. Assume that
  \begin{enumerate}[\upshape(1)]
    \item the canonical map
          $S \injects L^\square \to L_2^\square$ is injective, and
    \item its image does not meet the set~$Z'$ consisting of the classes of
          \[ 1 + \lambda^{2l-1},\quad 1 + \frac{\lambda^{l+2}}{1 + \lambda^2},\quad
             \sigma,\quad \sigma'
          \]
          in $L_2^\square$.
  \end{enumerate}
  Then $C'_l(\Q) = \{\infty, (1, 1), (1, -1)\}$.

  In particular, if $l = p$ is a prime, then the generalized Fermat equation
  $x^5 + y^5 = z^p$ has no nontrivial coprime integral solutions.
\end{proposition}

\begin{proof}
  Note that Lemmas \ref{L:im1GFE} and~\ref{L:im2GFE} imply that
  $Z' \cup \{1\}$ is the union of the sets~$Y$ occurring in Algorithm~\ref{Algo}
  when it is applied to the curve~$C'_l$,
  so the assumptions imply that the algorithm will not return {\sf FAIL}.
  (There cannot be any elements in $\Sel_2 C$ other than the images of
  the known points, since this would lead to a non-trivial intersection
  of~$Z'$ with the image of~$S$.)
  The set returned by the algorithm can contain at most one point in each
  $2$-adic residue disk. Since there are only three such disks, the known
  points must account for all rational points on~$C'_l$.
\end{proof}

Computing $\sigma$ and~$\sigma'$ for many values of~$l$,
it appears that their images in~$L_2^\square$ are represented uniformly by an infinite product
\[ (1 + \lambda^{l+2}) (1 + \lambda^{l+6}) (1 + \lambda^{l+8}) (1 + \lambda^{l+10})
               (1 + \lambda^{l+14}) (1 + \lambda^{l+18}) (1 + \lambda^{l+22}) \cdots ,
\]
but it is not obvious which rule is behind the sequence $(2, 6, 8, 10, 14, 18, 22, \ldots)$.
However, extending it further and consulting the OEIS~\cite{OEIS}
gives exactly one hit, namely A036554, the sequence of `numbers $n$ whose binary
representation ends in an odd number of zeros', i.e., such that $v_2(n)$ is odd.
So we propose the following.

\begin{conjecture}
  $\mu_2(\sigma)$ (and also $\mu_2(\sigma')$) is represented by
  \[ \prod_{n \ge 1, 2 \nmid v_2(n)} (1 + \lambda^{l+n})
      \sim 1 + \frac{\lambda^l}{\prod_{k \ge 1} (1 + \lambda^{2^k})} .
  \]
\end{conjecture}

We give a more concrete version of the criterion, following the considerations
of Remark~\ref{R:no2adic}.

\begin{corollary} \label{C:hyp2b}
  Assume that $l$ is prime and that $l^2 \nmid 2^{l-1} - 1$.
  Then a possible choice of the subgroup~$S$ in Proposition~\ref{P:hyp2a} is
  the subgroup of~$L(\{5\}, 2)$ consisting of elements mapping into the image
  of~$J'_l(\Q_5)$ in~$L_5^\square$.
  In fact, the resulting criterion is equivalent to what would be obtained
  by taking $S$ to be the image of~$\Sel_2 J'_l$.
\end{corollary}

\begin{proof}
  As in the case discussed in the preceding section, the assumption $l^2 \nmid 2^{l-1} - 1$
  implies that the Tamagawa number at~$l$ is~$1$, so that we can reduce to $\Sigma = \{2,5\}$.
  Furthermore, since $2$ is totally ramified in~$L$ and $L$ has odd degree,
  the norm of any element~$\alpha \in L^\times$ whose valuation with respect to
  the prime above~$2$ is odd will have odd $2$-adic valuation and cannot be a square.
  This lets us reduce to~$L(\{5\}, 2)$. Remark~\ref{R:no2adic} now shows that
  using~$S$ is equivalent to using~$\Sel_2 J'_l$ in the algorithm.
\end{proof}

We note that it is fairly easy to find~$S$, given $L(\{5\}, 2)$, since
the image of~$J'_l(\Q_5)$ in~$L_5^\square$ equals the
image of~$J'_l(\Q_5)[2]$, unless there are elements of order~$4$ in~$J'_l(\Q_5)$.
We can easily exclude this by checking that the images of an $\F_2$-basis
of~$J'_l(\Q_5)[2]$ are independent.

We carried out the computations necessary to test the criterion
of Proposition~\ref{P:hyp2a} in the version of Corollary~\ref{C:hyp2b}.
This results in the following.

\begin{theorem}
  For $7 \le p \le 53$ prime, we have (assuming GRH when $p \ge 23$)
  \[ C'_p(\Q) = \{\infty, (1, 1), (1, -1)\} . \]

  In particular, the generalized Fermat equation $x^5 + y^5 = z^p$ has only
  the trivial coprime integral solutions.
\end{theorem}

%%%%%%%%%%%%%%%%%%%%%%%%%%%%%%%%%%%%%%%%%%%%%%%%%%%%%%%%%%%%%%%%%%%%%%%%%%%%%%%%%%%%%%%%%

\section{An `elliptic Chabauty' example} \label{S:ellchab}

In this section, we apply our approach to `Elliptic Curve Chabauty'.
The curve in the following result comes up in the course of trying to
find all primitive integral solutions to the Generalized Fermat Equation
$x^2 + y^3 = z^{25}$. It is a hyperelliptic curve over~$\Q$ of genus~$4$;
it can be shown that the Mordell-Weil group of its Jacobian has rank~$4$
(generators of a finite-index subgroup can be found), so that Chabauty's
method does not apply directly to the curve.

\begin{theorem}
  Let $C$ be the smooth projective curve given by the affine equation
  \[ y^2 = 81 x^{10} + 420 x^9 + 1380 x^8 + 1860 x^7 + 3060 x^6 - 66 x^5
            + 3240 x^4 - 1740 x^3 + 1320 x^2 - 480 x + 69 .
  \]
  If GRH holds, then $C(\Q)$ consists of the two points at infinity only.
\end{theorem}

\begin{proof}
  As a first step, we compute the fake $2$-Selmer set as in~\cite{BruinStoll2009}.
  We obtain a one-element set (this requires local information only at
  the primes $2$, $3$ and~$29$). Using~\cite{Stoll2001}*{Lemma~6.3}, we then
  show that the points in~$C(\Q_2)$ whose image in~$L_2^\square/\Q_2^\square$
  is the image of the unique element of the fake $2$-Selmer set are those whose
  $x$-coordinate has $2$-adic valuation~$\le -3$. This set is the union of
  two half residue disks (the maximal residue disks contain the points~$P$
  such that $v_2(x(P)) \le -2$) that are mapped to each other by the hyperelliptic
  involution, so it is sufficient to consider just one of them, say the disk
  that contains $P_0 = \infty_9$, the point at infinity such that $(y/x^5)(P_0) = 9$.

  The splitting field of the polynomial~$f$ on the right hand side of the curve
  equation contains three pairwise non-conjugate subfields~$k$ of degree~$10$
  over which $f$ is divisible by a monic polynomial $g \in k[x]$ of degree~$4$.
  If $P$ is any rational point on~$C$, then it follows that $g(x(P))$ is a
  square in~$k$ (this is because the image of~$P$ in the fake $2$-Selmer set
  is the same as that of~$P_0$), so we obtain a point $(\xi, \eta) \in H(k)$
  with $\xi \in \Q$ (or $\xi = \infty$) where $H$ is the smooth projective
  curve given by $y^2 = g(x)$. We can parameterize the image of the residue disk
  around~$P_0$ by a pair of Laurent series $\bigl((2t)^{-1}, \sqrt{g((2t)^{-1})}\bigr)$
  (where we can take the square root to have leading term $t^{-2}/4$).
  Since $H(k)$ contains the two points at infinity, $H$ is isomorphic to
  an elliptic curve~$E$ over~$k$; we take the isomorphism so that it
  sends $\infty_1 \in H(k)$ to the origin of~$E$. We then obtain a Laurent
  series~$\xi(t) \in k\lrs{t}$ that gives the $x$-coordinate of the image on~$E$
  of the point whose parameter is~$t$.

  For the following, we take $k$ to be the field generated by a root of
  \[ x^{10} + 75 x^6 - 50 x^5 + 100 x^3 + 625 x^2 + 1250 x + 645; \]
  the polynomial~$g$ and the curves $H$ and~$E$ are taken with respect to this field.
  We next compute the $2$-Selmer group of~$E$ over~$k$. There is exactly one
  point of order~$2$ in~$E(k)$, which means that we have to work with a quadratic
  extension of~$k$. This is where we use GRH, which allows us to find the relevant
  arithmetic information for this field of degree~$20$ faster. The Selmer group
  has $\F_2$-dimension~$6$ (so the bound for the rank of~$E(k)$ is~$5$).
  We check that it injects into~$E(k_2)/2 E(k_2)$, where $k_2 = k \otimes_{\Q} \Q_2$;
  note that this splits as a product of two extensions of~$\Q_2$, both of ramification
  index~$2$ and one of residue class degree~$1$, the other of residue class degree~$4$.

  In the context of our method, we consider the curve that is the (desingularization of) the
  curve over~$\Q$ in~$A = R_{k/\Q} E$ (the latter denotes the Weil restriction of scalars)
  that corresponds to the set of points on~$H$ whose $x$-coordinate is rational.
  We have $E(k_2) \cong A(\Q_2)$, so we can use arithmetic on~$E$ over~$k$ and its
  completions for the computations. We check that $n_\tors = 1$ (the map from $E(k_2)[2]$
  to $E(k_2)/2E(k_2)$ is injective). By Lemma~\ref{L:qconstnear0}, a suitable
  value of~$m$ is $m = 4$, provided $5 \ge n_4$ in the notation of the lemma.
  Note that in this situation halving points is easy, since doubling a point corresponds
  to an explicit map of degree~$4$ on the $x$-coordinate.
  If $P$ is in our half residue disk, then $i(P) + T$ (where $T$ is the point of order~$2$
  in $E(k) = A(\Q)$) is not divisible by~$2$, and its image in $A(\Q_2)/2A(\Q_2)$
  is the same as that of~$T$. So we only have to determine $q(i(P))$ for a suitable
  selection of points~$P$ in our disk. We write $P_\tau$ for the point corresponding
  to the parameter $\tau \in 2 \Z_2$.

  We compute $Y_\tau = q(i(P_\tau))$ for $\tau \in \{\pm 4, \pm 8, \pm 16\}$.
  Note that this can be done solely in terms of the $x$-coordinate~$\xi(\tau)$.
  We find that $\nu(i(P_\tau)) = v_2(\tau) - 1$ for these values and that
  $Y_{-\tau} = Y_\tau$. This implies that $n_m = m - 1$ for $2 \le m \le 4$
  and that $q$ of the disk in question is the union $Y_4 \cup Y_8 \cup Y_{16}$.
  In fact, it turns out that this union is equal to~$Y_8$, and we verify that
  $Y_8$ meets the image of the $2$-Selmer group only in the image of the global torsion.
  By Theorem~\ref{T:key}, this then implies that there can be no other point
  than~$P_0$ in our (half) residue disk. The claim follows.
\end{proof}

We note that the two other possible choices of~$k$ also lead to elliptic curves
with a $2$-Selmer rank of~$5$ (this is unconditional for one of the choices,
where the curve~$E$ happens to have full $2$-torsion over~$k$), but for these
other fields, the condition that the image of the disk under~$q$ meets the image
of the Selmer group only in the image of the global torsion is not satisfied.
We also remark that we have been unable to find five independent points in~$E(k)$
(for any of the three possible choices of $k$ and~$E$),
so that we could not apply the standard Elliptic Curve Chabauty method.

Another application of our `Selmer group Chabauty' approach in the setting
of Elliptic Curve Chabauty was made in~\cite{FreitasNaskreckiStoll}.
We use this to show that there are no unexpected points on the elliptic
curve~$X_0(11)$ defined over certain number fields of degree~$12$ and such
that the image under the $j$-map is in~$\Q$. This is a vital step in the proof
that the only nontrivial primitive integral solutions of the Generalized
Fermat Equation $x^2 + y^3 = z^{11}$ are $(x,y,z) = (\pm 3, -2, 1)$.
The situation is similar to what happens for the example presented here:
we can compute the $2$-Selmer group of~$X_0(11)$ over the fields of interest,
but we are unable to produce enough independent points to meet the upper
bound on the rank, so we cannot apply the standard method.

%%%%%%%%%%%%%%%%%%%%%%%%%%%%%%%%%%%%%%%%%%%%%%%%%%%%%%%%%%%%%%%

\end{document}